\documentclass[11pt, oneside]{article}   	
\usepackage[utf8]{inputenc}
\usepackage{geometry}                		
\usepackage[parfill]{parskip}    		
\usepackage{graphicx}				
\usepackage{amsmath}
\usepackage{amssymb}
\usepackage{amsthm}

\usepackage{dsfont}

\usepackage{esint}

\usepackage{mathrsfs}

\usepackage[english]{babel}
\usepackage{fancyhdr}
 
\theoremstyle{definition}
\newtheorem{definition}{Definition}[section]

\theoremstyle{theorem}
\newtheorem{theorem}{Theorem}[section]

\theoremstyle{theorem}
\newtheorem*{theorem*}{Theorem}

\theoremstyle{theorem}
\newtheorem{lemma}{Lemma}[section]

\theoremstyle{theorem}
\newtheorem{proposition}{Proposition}[section]

\theoremstyle{theorem}

\theoremstyle{corollary}
\newtheorem{corollary}{Corollary}[section]

\theoremstyle{remark}
\newtheorem{remark}{Remark}[section]

\usepackage{polynom}

\usepackage{tikz}
\usepackage{tikz-cd}
\usetikzlibrary{matrix,arrows,decorations.pathmorphing}

\usepackage{cite}

\usepackage{caption}

\DeclareMathOperator{\G}{\mathbb{G}}
\DeclareMathOperator{\X}{\mathbb{X}}
\DeclareMathOperator{\Y}{\mathbb{Y}}
\DeclareMathOperator{\x}{\mathbf{x}}
\DeclareMathOperator{\y}{\mathbf{y}}
\DeclareMathOperator{\z}{\mathbf{z}}
\DeclareMathOperator{\w}{\mathbf{w}}

\DeclareMathOperator{\Rect}{Rect}
\DeclareMathOperator{\Pent}{Pent}
\DeclareMathOperator{\Hex}{Hex}
\DeclareMathOperator{\Int}{Int}
\DeclareMathOperator{\Long}{Long}

\theoremstyle{theorem}
\newtheorem*{integer}{Theorem \ref{integer}}

\theoremstyle{theorem}
\newtheorem*{skein}{Theorem \ref{skein}}

\theoremstyle{theorem}
\newtheorem*{alt}{Theorem \ref{alt}}

\theoremstyle{theorem}
\newtheorem*{tor}{Theorem \ref{tor}}

\begin{document}

\title{Combinatorial Proofs of Properties of Double-Point Enhanced Grid Homology}
\author{Ollie Thakar}
\maketitle

\abstract{We provide a purely combinatorial proof of a skein exact sequence obeyed by double-point enhanced grid homology. We also extend the theory to coefficients over $\mathbb{Z},$ and discuss alternatives to the Ozsv\'ath-Szab\'o $\tau$ invariant.}

\tableofcontents

\section*{Organization of the Paper}

Knot Floer homology ([OS], [R03]) is a powerful knot invariant defined similarly to Heegaard Floer homology, exhibiting many desirable properties such as detecting Seifert genus and fiberedness sharply. The definition and subsequent computations rely heavily on holomorphic geometry. However, in [MOS], Ciprian Manolescu, Peter Ozsv\'{a}th, and Sucharit Sarkar define a combinatorial knot invariant known as grid homology, denoted $GH^-$, using a grid diagram of a knot (see also [MOST, OSS].) Grid homology provides a more concrete way to compute knot Floer homology. Indeed, this invariant is in fact isomorphic to knot Floer homology, but many of its topological virtues can be proven purely combinatorially, such as lower bounds it provides on unknotting number and 4-ball genus.

A slight variation on their definition, by Robert Lipshitz, is known double-point enhanced grid homology, which we notate $GHL$ (see [L] and [OSS Chapter 5].) It remains unknown whether the double-point enhanced grid homology actually encodes new information beyond what is accessible to regular grid homology. Specifically, it is conjectured that for all knots $K,$ we have $GH^-(K)[v]\cong GHL(K)$ as bigraded $\mathbb{F}[U, v]$-modules (where $\mathbb{F}$ is the field of 2 elements.) While we do not settle this question in the paper, we do prove two properties of double-point enhanced grid homology that we already knew to be obeyed by grid homology. Earlier work of Timothy Ratigan, Joshua Wang, and Luya Wang finds a purely combinatorial proof that $GHL$ is indeed a knot invariant, and conjectures both properties of double-point enhanced grid homology that we prove in this paper as Theorems \ref{integer} and \ref{skein} (see [RWW].)

The first section of the paper reviews grid homology as defined in [OSS], with a few theorems that are relevant to our later pursuits. Sections 3 and 4 define the double-point enhanced grid homology of a knot and provide useful definitions and lemmas. The goal of the paper is to prove two important theorems. The first, stated below, proves that double-point enhanced grid homology admits an integer-coefficient version which is a knot invariant. We will prove a slightly stronger statement in Section 5 along the way.

\begin{integer}
For each grid diagram of a knot, there exists a homology group $GHL_S(\G; \mathbb{Z}),$ which, as a bigraded $\mathbb{Z}[U, v]$-module, is also a knot invariant. Furthermore, $GHL_S(\G; \mathbb{Z})$ is the homology of a chain complex, which when we take the homology of the mod 2 version gives us the double-point enhanced grid homology $GHL(\G).$
\end{integer}

The subscript $S$ in this notation refers to a \emph{sign assignment,} a function on certain rectangles in $\G.$ We will initially define $GHL_S(\G; \mathbb{Z})$ in terms of one such function $S$ and then prove it is invariant under change in $S.$

The second theorem we prove shows that this homology $GHL_S(\G; \mathbb{Z}),$ which we call integral double-point enhanced grid homology, obeys a skein exact sequence. We extend $GHL$ to a link invariant $cGHL_m(L, a).$ Omitting the subscript $S$ and $\mathbb{Z}$ from the notation, we may state the theorem:

\begin{skein}
Let $(L_+, L_-, L_0)$ be an oriented skein triple, with $\ell$ and $\ell_0$ the number of components of $L_+$ and $L_0$ respectively. If $\ell_0=\ell+1,$ then there is a long exact sequence where the maps below fit together to be homomorphisms of $\mathbb{Z}[U, v]$-modules:

$$\to cGHL_m(L_+, s) \to cGHL_m(L_-, s) \to cGHL_{m-1}(L_0, s) \to cGHL_{m-1}(L_+, s) \to$$

Let $J$ be the 4-dimensional bigraded abelian group $J\cong \mathbb{Z}^4$ with one generator in bigrading $(0,1),$ one generator in bigrading $(-2, -1),$ and two generators in bigrading $(-1, 0).$

If $\ell_0=\ell-1,$ then there is a long exact sequence where the maps below fit together to be homomorphisms of $\mathbb{Z}[U, v]$-modules:

$$\to cGHL_m(L_+, s) \to cGHL_m(L_-, s) \to cGHL_{m-1}(L_0, s)\otimes J \to cGHL_{m-1}(L_+, s) \to$$
\end{skein} 

Finally, for the last two sections we return to $\mathbb{F}$ coefficients for simplicity. Section 7 presents some more concrete invariants that can be extracted out of double-point enhanced grid homology, and their potential use in proving that the two invariants do not encode different information. Section 8 computes the double-point enhanced grid homology of alternating knots and torus knots over the field of 2 elements using a spectral sequence, and shows that in these cases the conjecture that $GH^-(K)[v]\cong GHL(K)$ holds with one caveat. The spectral sequence loses information about the $v$ action, hence these two below theorems are stated under a weaker-than-ideal condition.

\begin{alt}
If $K$ is a quasi-alternating knot, then $GHL(K) \cong GH^-(K)[v]$ as bigraded $\mathbb{F}[U]$-modules.
\end{alt}

\begin{remark}
The family of quasi-alternating knots is a family that contains the alternating knots but is strictly larger (see [OSS Chapter 10].)
\end{remark}

\begin{tor}
If $K$ is a torus knot, then $GHL(K) \cong GH^-(K)[v]$ as bigraded $\mathbb{F}[U]$-modules.
\end{tor}

\section{Acknowledgements}

I would like to thank Peter Ozsv\'{a}th for his advice and support throughout the project, as well as the idea for the project itself. Additionally, I would like to thank Zolt\'{a}n Szab\'{o}, Isabella Khan, Matthew Kendall, Luya Wang, and Joshua Wang --- who in particular noticed that Theorem \ref{alt} holds for quasi-alternating knots rather than only for alternating knots --- for helpful correspondences.

\section{Background on Grid Homology}

We begin with a brief summary of grid homology, a knot invariant defined in [OSS]. For this section, fix a knot $K$. A \emph{grid diagram} is an $n$-by-$n$ grid of squares such that there is one $X$ and one $O$ in each row and each column. We may retrieve a link from a grid diagram as follows. For each row and each column of squares, draw a line segment connecting the $X$ to the $O$ within that row or column respectively. Specify further that whenever two such segments intersect, the vertical segment crosses over the horizontal segments. Then, it is clear that the union of all these segments is a planar link diagram. [OSS Theorem 3.1.3] guarantees that there exists a grid diagram representing any link, in particular $K.$ Let $\G$ be such a diagram.

Denote by $\X$ the set of all the center points of the $X$-marked squares, and likewise denote by $\mathbb{O}$ the set of all the center points of the $O$-marked squares.

We consider $\G$ to be a fundamental domain of a torus $\mathbb{T}$ constructed by gluing opposite sides of $\G.$ It is clear that different fundamental domains of this torus represent isotopic links. Fix a coordinate system on a fundamental domain corresponding to the cardinal directions North, South, East, and West.

Call the horizontal circles formed by the edges of the grid of squares as $\alpha_1, \dots, \alpha_n,$ moving further north, and the vertical circles formed by the edges of the grid of squares as $\beta_1, \dots, \beta_n,$ moving further east.

\begin{definition}
A \emph{grid state} $\x$ of $\G$ is a set of $n$ points on $\mathbb{T}$ such that $|\x\cap\alpha_i|=1$ for all $i\in\{1,\dots, n\}$ and $|\x\cap\beta_i|=1$ for all $i\in\{1,\dots, n\}.$ (In other words, $\x$ is a set of $n$ intersection points of the $\alpha$ and $\beta$ curves, such that each curve is represented once.)

We denote the set of all grid states of $\G$ by $\mathbf{S}(\G).$
\end{definition}

\begin{definition}
A \emph{rectangle} in $\G.$ is an embedding of the closed unit disk $D^2$ into $\G$ such that $\partial D^2$ gets mapped into the union of the $\alpha$ and $\beta$ curves. Let $$A := \partial r\cap(\alpha_1\cup\dots\cup\alpha_n).$$ Then $A$ is a 1-manifold with boundary consisting of two line segments, and has an orientation induced from $r$ by moving clockwise around the boundary $\partial r.$. We say the rectangle $r$ connects two grid states $\x,\y\in\mathbf{S}(\G)$ (or goes from $\x$ to $\y$) if $\partial A = \y - \x,$ and such that the points in $\y\cap\partial A$ inherit a positive orientation from $A.$

The set of all rectangles from $\x$ to $\y$ is denoted by $\Rect(\x, \y).$
\end{definition}

We wish to create two functions $M$ and $A$ from $\mathbf{S}(\G)$ to $\mathbb{Z},$ which we define as follows.

\begin{definition}
Let $P$ and $Q$ be sets of finitely many points in a fundamental domain for $\G,$ which we may embed in $\mathbb{R}^2$ with standard Cartesian coordinates as the rectangle $[0, n)\times [0, n)$ such that each square in $\G$ is a unit square with integral coordinates for its corners. Then, we define $\mathcal{I}(P, Q)$ to be the number of pairs of points $(p_1, p_2)\times(q_1, q_2)\in P\times Q$ satisfying $p_1<q_1$ and $p_2<q_2.$ Now, let $$\mathcal{J}(P, Q) = \frac12(\mathcal{I}(P, Q) + \mathcal{I}(Q, P)).$$

Then, we let $$M(\x) := \mathcal{J}(\x, \x) - 2\mathcal{J}(\x, \mathbb{O}) + \mathcal{J}(\mathbb{O}, \mathbb{O}) + 1,$$ $$M_{\X}(\x) := \mathcal{J}(\x, \x) - 2\mathcal{J}(\x, \mathbb{X}) + \mathcal{J}(\mathbb{X}, \mathbb{X}) + 1,$$ and $$A(\x) = \frac12(M(\x) - M_{\X}(\x)) - \frac{n-1}{2}.$$ We call $M(\x)$ the \emph{Maslov grading} and $A(\x)$ the \emph{Alexander grading}, and we call the pair $(M, A)$ the \emph{bigrading} of $\x.$
\end{definition}

The below proposition is greatly helpful to our future ventures:

\begin{proposition}
\begin{itemize}
	\item Both $M$ and $A$ are integral-valued functions (note it is only clear from their definitions that they are half-integral valued.)
\end{itemize}
Suppose $\x$ and $\y$ are two grid states with some rectangle $r\in\Rect(\x, \y).$ Then, their Maslov and Alexander gradings are related by the following formulas:
\begin{itemize}
	\item $M(\x) - M(\y) = 1 - 2|r\cap\mathbb{O}| + 2|\x\cap\Int(r)|$
	\item $A(\x) - A(\y) = |r\cap\mathbb{X}| - |r\cap\mathbb{O}|$
\end{itemize}
\end{proposition}

The proof is found in [OSS Section 4.3], and is elementary but rather long.

\begin{remark}
For the remainder of this paper, let $\mathbb{F}$ represent the field of 2 elements $\mathbb{Z}/2\mathbb{Z}.$
\end{remark}

\begin{definition}
Let a \emph{domain} $\psi$ in $\G$ be any formal $\mathbb{Z}$-linear combination of squares in $\G$ (which may be defined as the closures of the connected components of $\G - (\alpha_1\cup\dots\cup\alpha_n\cup\beta_1\cup\dots\cup\beta_n).$) 

Again, the boundary of a domain inherits a clockwise orientation. We say a domain connects two grid states $\x$ and $\z$ in $\mathbf{S}(\G)$ if $$\partial(\partial \psi\cap(\alpha_1\cup\dots\cup\alpha_n)) = \z - \x,$$ with the $\z$-portions inheriting a positive orientation and the $\x$-portions a negative orientation. Let the set of all domains from $\x$ to $\z$ be denoted as $\pi(\x, \z).$

Let the multiplicity of the domain $\psi$ at the point $p\in\mathbb{G}$ be the coefficient of the square containing $p$ in the expression for $\psi$ as a combination of squares (we will only be considering multiplicities at points $p$ not on the boundary of any square.) Let $\psi(p)$ be the multiplicity of $\psi$ at the point $p.$

We say a domain is decomposed as a juxtaposition of two rectangles $r_1\in\Rect(\x, \y)$ and $r_2\in\Rect(\y, \z)$, and write $\psi = r_1*r_2$, if $\psi\in\pi(\x, \z)$ and the multiplicities satisfy the following formula: $$\psi(p)=r_1(p)+r_2(p)$$ for all $p\in\G$.
\end{definition}

We are now ready to define grid homology.

\begin{definition}
Let $\G$ be a grid diagram. We define the chain complex $GC^-(\G)$ to be the free $\mathbb{F}[V_1,\dots, V_n]$-module generated by the grid states of $\G,$ with $V_1^{k_1}\dots V_n^{k_n}\x$ having bigrading $$(M(\x)-2k_1-\dots-2k_n, A(\x)-k_1-\dots-k_n).$$

Let $\partial_0:GC^-(\G)\rightarrow GC^-(\G)$ be given as $$\partial_0\x = \sum_{\y\in\mathbf{S}(\G)}\sum_{r\in\Rect(\x, \y)} V_1^{O_1(r)}\dots V_n^{O_n(r)}\y.$$ [OSS Chapter 4] demonstrates that $\partial_0$ is indeed a differential, that is it is a homogeneous map of bidegree $(-1, 0)$ and $\partial_0^2=0,$ hence the homology of the chain complex $(GC^-(\G), \partial_0)$ is well-defined. We denote this homology as $GH^-(\G).$ 
\end{definition}

[OSS Chapter 5] proves that the action of each $V_i$ is identical on the level homology, and calling this action by $U,$ that $GH^-(\G)$ is in fact an invariant of the knot $K$ as a bigraded $\mathbb{F}[U]$-module; we may write it as $GH^-(K).$

\begin{definition}
The homology of the complex $(GC^-(\G)/(V_n = 0), \partial_0)$ is referred to as $\widehat{GH}(\G).$
\end{definition}

[OSS Chapter 5] proves that this homology is also a knot invariant $GH^-(\G)$ is in fact an invariant of the knot $K$ as a bigraded $\mathbb{F}$-vector space.

\subsection{Grid Moves}

We would like a way to relate any two grid representations of the same (oriented) link. We define three types of grid moves: commutations, switches, and stabilizations.

\begin{definition}
Consider two adjacent rows (resp. columns) of a grid diagram in a fundamental domain, and draw the (closed) line segments joining the $X$- and $O$- markings in each row (resp. column). Project these two line segments onto the horizontal (resp. vertical) axis. If either (1) the two projected line segments have disjoint supports or (2) one of the projected line segments completely contains the other, then swapping the two adjacent rows (resp. columns) is called a \emph{commutation}.

If the two projected line segments share a vertex, then then swapping the two adjacent rows (resp. columns) is called a \emph{switch}.
\end{definition}

Here, we picture a commutation:

\includegraphics[scale=0.4]{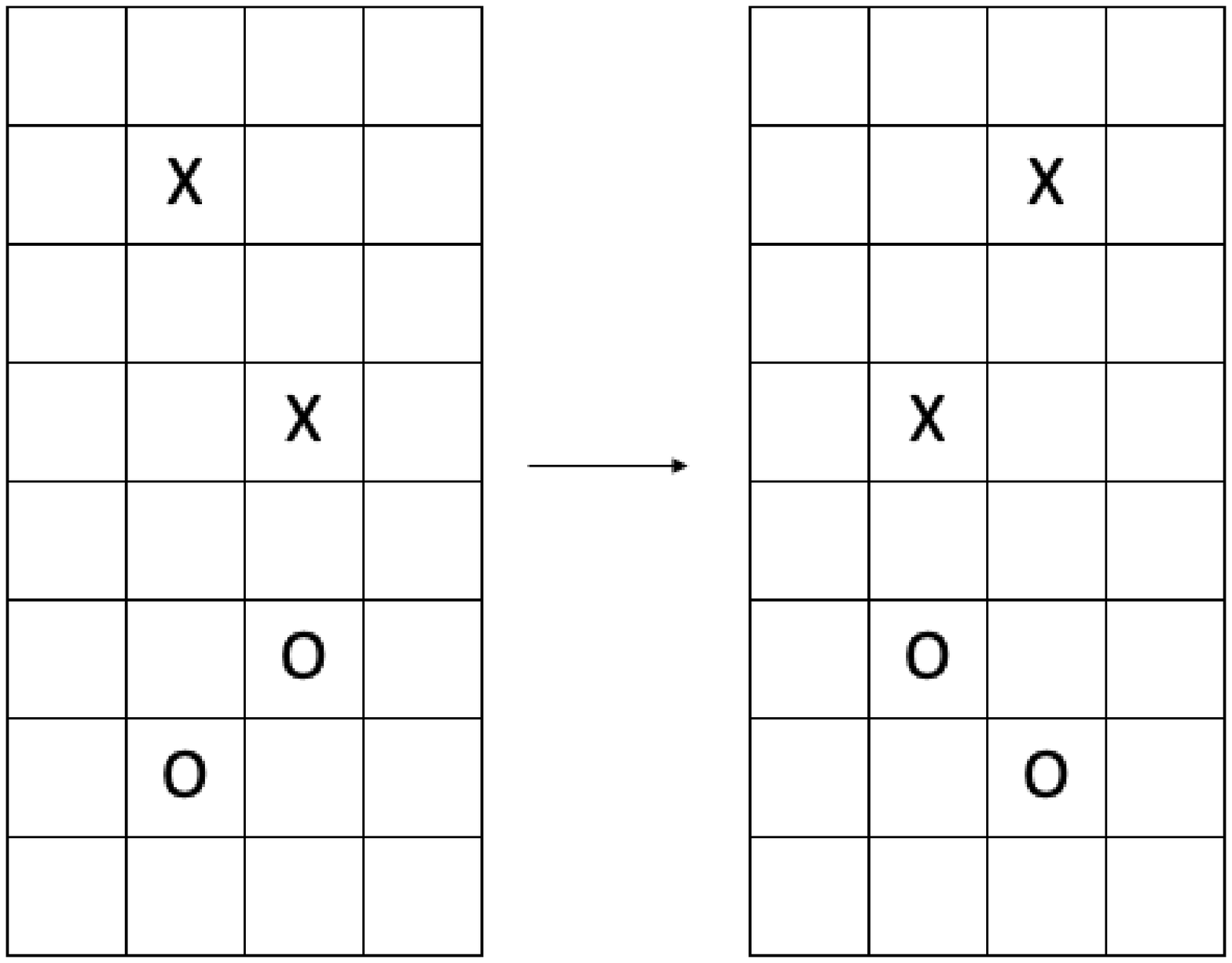}

\begin{definition}
Consider a square marked with an $X$ (resp. an $O$). Choose one of the following four directions, $NE, NW, SE, SW.$ Subdivide the row and column containing the $X$ (resp. $O$) so that $\G$ is now an $n+1$-by-$n+1$ grid diagram, and the square formerly containing the $X$ (resp. $O$) is now a 2-by-2 grid. Replace the $X$ (resp. $O$) with two $X$'s in the diagonal of this 2-by-2 grid that does not contain the chosen direction, and an $O$ (resp. $X$) in a third square of this 2-by-2 grid such that the unmarked square is the one corresponding to the chosen direction. This operation is known as a \emph{stabilization} of type $X:\text{direction}$ (resp. $O:\text{direction}$). Its inverse is known as a \emph{destabilization} of the corresponding type.
\end{definition}

Here, we picture a stabilization of type $X:SW$:

\includegraphics[scale=0.4]{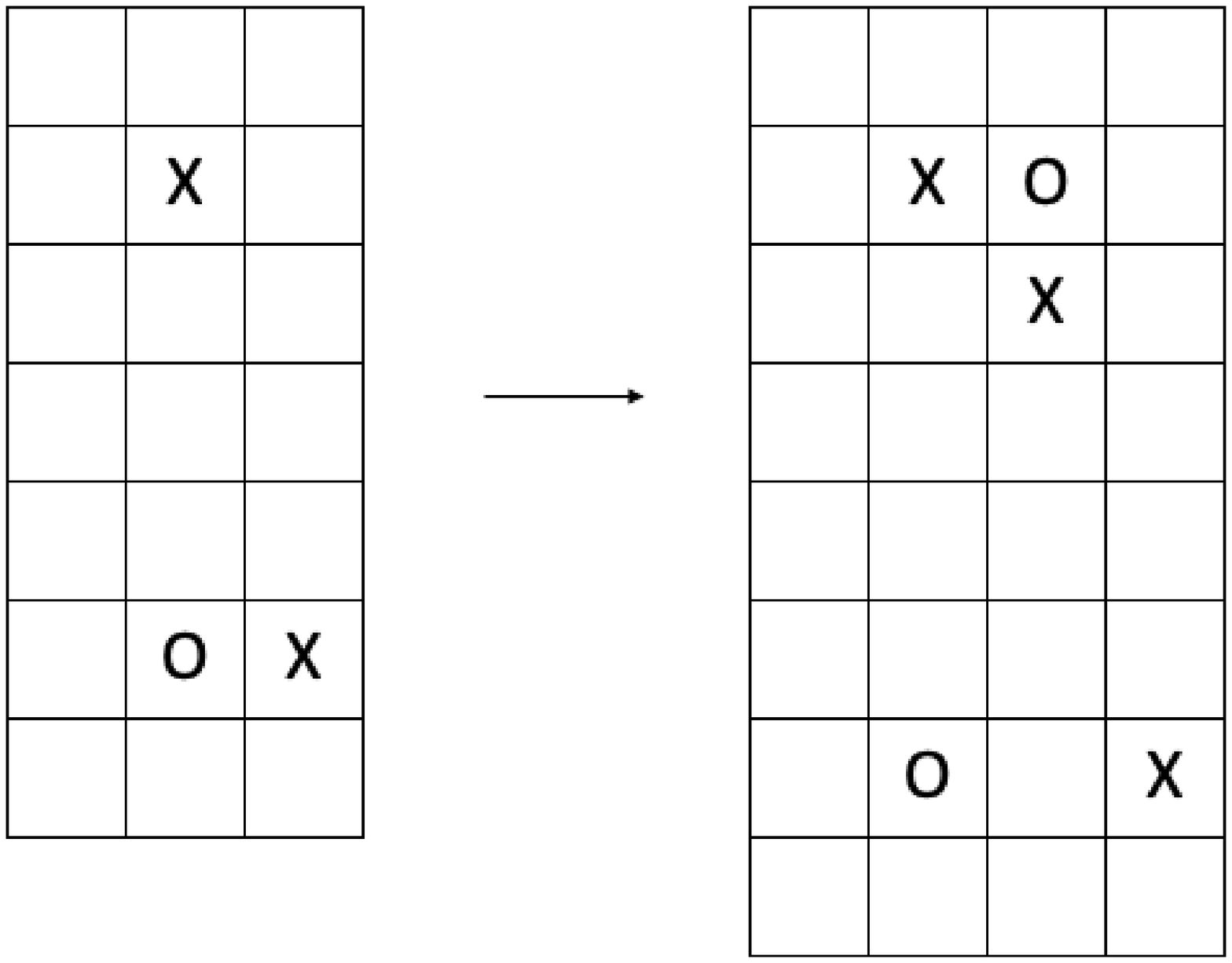}

The following theorem will prove extremely useful in showing invariance of double-point enhanced grid homology:

\begin{theorem}[generalized from Cromwell, see \text{[OSS Corollary 3.2.3]}]\label{cromwell}
Any two grid diagrams of the same (oriented) knot are related by a finite sequence of commutations, switches, and stabilizations and destabilizations of the form $X:SW.$
\end{theorem}

\section{Double-Point Enhanced Grid Homology Notation}

\begin{definition}
Fix an $n$-by-$n$ grid diagram $\G$. We define a bigraded chain complex of free $\mathbb{F}[V_1,\dots, V_n, v]$-modules $GCL(\G)$ as follows. As a bigraded module, $GCL(\G) = GC^-(\G)[v]$, where if $\xi\in GC^-(\G)$ is homogeneous of bidegree $(M, A),$ then $v^k\xi$ is homogeneous of bidegree $(M+2k, A).$

We define a differential $$\partial\xi := \sum_{n=0}^\infty v^n\partial_n\xi,$$ where $\partial_n$ is defined on grid states by: $$\partial_n\x= \sum_{\y\in\mathbf{S}(\G)}\sum_{\{r\in\Rect(\x, \y)\big| r\cap\X = \emptyset, |\Int(r)\cap\x| = n\}} V_1^{O_1(r)}\dots V_n^{O_n(r)}\y,$$ and extends by linearity.
\end{definition}

\begin{proposition}
The map $\partial$ is indeed a differential, that is, $\partial^2=0.$
\end{proposition}

We shall prove a more general version of this proposition later, see \ref{diff}. For now, we take this for granted, and let $GHL(\G)$ denote the homology of the chain complex $(GCL(\G), \partial).$

For any toroidal grid diagram on a torus $\mathbb{T}$, we may consider the universal cover of the torus, which we identify with $\mathbb{R}^2$ and its standard $(x,y)$ Cartesian coordinates. Here, lifts of the $\alpha_i$ and $\beta_j$ curves, which we may call $\tilde{\alpha}_i,$ $\tilde{\beta}_j$ respectively, are the straight lines $y=n$ and $x=m$ as $n,m$ range over $\mathbb{Z}.$

\begin{definition}
Consider a rectangle $R$ of width one in $\mathbb{R}^2$ whose sides lie along the $\tilde{\alpha}_i,$ $\tilde{\beta}_j$ lines, and such that the projection of $R$ onto $\mathbb{T}$, which we call $r$, has multiplicity $2$ in at least one point, and multiplicity 1 in at least one point. Then, we call $r$ a \emph{long rectangle}. 

Note that $r$ is a domain in $\G,$ and it connects grid states analogously to ordinary rectangles. We denote by $\Rect^*(\x, \y)$ the set of rectangles and long rectangles from grid state $\x\in\mathbf{S}(\G)$ to $\y\in\mathbf{S}(\G).$
\end{definition}

\begin{definition}
We define the function $\mathcal{T}:\Rect^*(\x, \y)\rightarrow\mathbb{Z}_{\geq0}$ as follows. Let $r\in\Rect^*(\x, \y).$  If $r$ is long, then $\mathcal{T}(r)=1,$ and if $r$ is not long, then $\mathcal{T}(r) = |\Int(r)\cap\x|.$
\end{definition}

\begin{remark}
Note that we can now rewrite the differential $\partial$ more compactly as $$\partial\x= \sum_{\y\in\mathbf{S}(\G)}\sum_{\{r\in\Rect(\x, \y)\big| r\cap\X = \emptyset\}} v^{\mathcal{T}(r)}V_1^{O_1(r)}\dots V_n^{O_n(r)}\y.$$
\end{remark}

\begin{definition}
Let $\x, \z\in\mathbf{S}(\G),$ and suppose $\psi\in\pi(\x, \z)$ is a domain, with decomposition $\psi = r_1*r_2$ for $r_1\in\Rect^*(\x,\y)$ and $r_2\in\Rect^*(\y,\z).$ The \emph{degree} of the decomposition, which we will notate as $\deg(r_1, r_2),$ is defined as the sum: $$\deg(r_1, r_2) = \mathcal{T}(r_1)+\mathcal{T}(r_2).$$
\end{definition}

\begin{definition}
For a rectangle or long rectangle $r\in\Rect^*(\x, \y)$, the \emph{incoming corners} are precisely the members of $\partial r\cap\x,$ and the \emph{outgoing corners} are precisely the members of $\partial r\cap\y.$ 
\end{definition}

\subsection{For Commutation/Switch Invariance}

To prove invariance of grid homology under commutation and switch moves, and also to prove the skein exact sequence, we will require superimposing two grid diagrams $\G$ and $\G'$ differing by a commutation or switch as pictured in the below picture of the relevant portion of this superimposed diagram:

\includegraphics[scale=0.4]{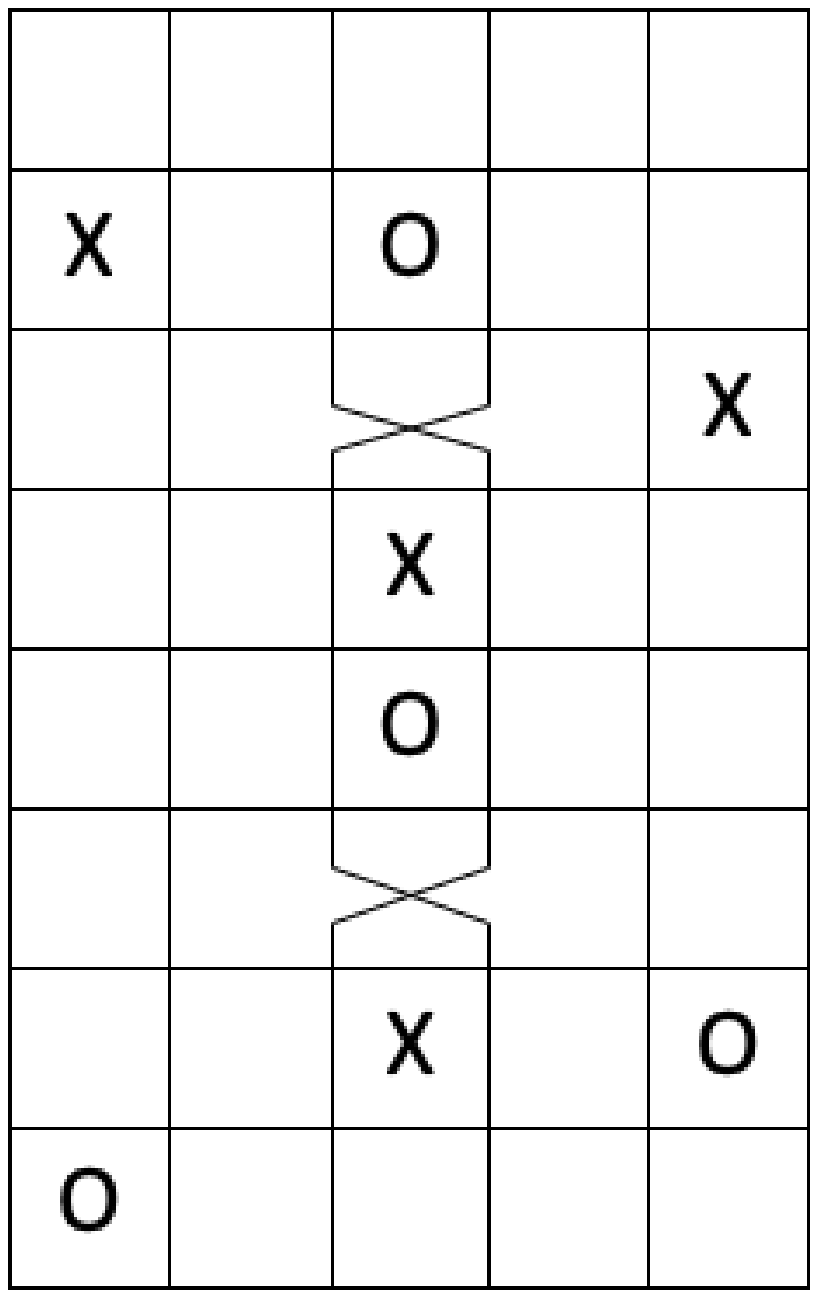}

We call $\beta_i$ the curved circle belonging to $\G$ and $\gamma_i$ the curved circle belonging to $\G'.$ Let $a$ and $b$ be the two intersections of $\beta_i$ and $\gamma_i,$ with $a$ at the southern end of the bigon containing $\beta_i$ as its western boundary.

\begin{remark}\label{bigon}
Very importantly, we may always assume that each bigon contains at least one $X$-marking in it. We will be counting regions that are forbidden from intersecting $X$-markings, hence this assumption will markedly simplify our below analysis.
\end{remark}

\begin{definition}
(Modified from [OSS] Definition 5.1.1). A \emph{pentagon} from $\x\in\mathbf{S}(\G)$ to $\y'\in\mathbf{S}(\G')$ is an embedded disk $p$ in the torus whose boundary is the union of five arcs, each of which lies on an $\alpha_j, \beta_j,$ or $\gamma_i$ curve, such that: (1) four of the corners of $p$ are in $\x\cup\y',$ (2) at each corner $x$ of $p,$ exactly one of the four quadrants of a small disk surrounding $x$ has multiplicity 1 and the other 3 have multiplicity 0, and (3) $$\partial(\partial p\cap(\alpha_1\cup\dots\cup\alpha_n)) = \y' - \x.$$

Let $P$ be an embedded disk in the universal cover of the torus satisfying conditions (1), (2), and (3), satisfying two extra conditions: (4) that $P$ has width one, and (5), that the projection of $P$ onto the torus, which we call $p$, has multiplicity $2$ in at least one point, and multiplicity 1 in at least one point. Then, we call $p$ a \emph{long pentagon} from $\x$ to $\y'.$

Let $\Pent(\x, \y')$ denote the set of pentagons from $\x$ to $\y',$ and $\Pent^*(\x, \y')$ denote the set of pentagons and long pentagons from $\x$ to $\y'.$
\end{definition}

\begin{definition}
(Modified from [OSS] Definition 5.1.5). A \emph{hexagon} from $\x\in\mathbf{S}(\G)$ to $\y\in\mathbf{S}(\G)$ is an embedded disk $h$ in the torus whose boundary is the union of six arcs, each of which lies on an $\alpha_j, \beta_j,$ or $\gamma_i$ curve, such that: (1) four of the corners of $h$ are in $\x\cup\y',$ and the other two corners are at $a$ and $b$, (2) at each corner $x$ of $h,$ exactly one of the four quadrants of a small disk surrounding $x$ has multiplicity 1 and the other 3 have multiplicity 0, and (3) $$\partial(\partial h\cap(\alpha_1\cup\dots\cup\alpha_n)) = \y - \x.$$

Let $H$ be an embedded disk in the universal cover of the torus satisfying conditions (1), (2), and (3), satisfying two extra conditions: (4) that $P$ has width one, and (5), that the projection of $H$ onto the torus, which we call $h$, has multiplicity $2$ in at least one point, and multiplicity 1 in at least one point. Then, we call $h$ a \emph{long pentagon} from $\x$ to $\y.$

Let $\Hex(\x, \y)$ denote the set of pentagons from $\x$ to $\y,$ and $\Hex^*(\x, \y)$ denote the set of pentagons and long pentagons from $\x$ to $\y.$
\end{definition}


\begin{definition}
We define the function $\mathcal{T}:\Pent^*(\x, \y')\rightarrow\mathbb{Z}_{\geq0}$ as follows. Let $p\in\Pent^*(\x, \y).$  If $p$ is long, then $\mathcal{T}(p)=1,$ and if $p$ is not long, then $\mathcal{T}(p) = |\Int(p)\cap\x|.$ Similarly, we define the function $\mathcal{T}:\Hex^*(\x, \y')\rightarrow\mathbb{Z}_{\geq0}$ as follows. Let $h\in\Hex^*(\x, \y).$  If $h$ is long, then $\mathcal{T}(h)=1,$ and if $h$ is not long, then $\mathcal{T}(h) = |\Int(h)\cap\x|.$
\end{definition}

\begin{definition}
Let $\x, \z\in\mathbf{S}(\G)\cup\mathbf{S}(\G'),$ and suppose $\psi\in\pi(\x, \z)$ is a domain, with decomposition $\psi = r_1*r_2$ for $r_1\in\Rect^*(\x,\y)\cup\Pent^*(\x,\y)\cup\Hex^*(\x,\y)$ and $r_2\in\Rect^*(\y,\z)\cup\Pent^*(\y,\z)\cup\Hex^*(\y,\z).$ The \emph{degree} of the decomposition, which we will notate as $\deg(r_1, r_2),$ is similarly defined as the sum: $$\deg(r_1, r_2) = \mathcal{T}(r_1)+\mathcal{T}(r_2).$$
\end{definition}

\section{Rectangle Decomposition Lemmas}

This section contains many useful combinatorial lemmas that will expedite the proofs of the later theorems tremendously.

We will set some consistent notation throughout this section. Fix a grid diagram $\G.$ Let $\x, \z\in\mathbf{S}(\G)$ and let $\psi\in\pi(\x, \z)$ be a fixed domain.

Observe that if $\psi$ admits at least one decomposition $\psi = r_1*r_2$ where $r_1$ and $r_2$ are either rectangles or long rectangles, then we must have $|\x - \x\cap \z| = 0,$ 3, or 4, simply because the initial and final grid states of each rectangle differ by exactly 2 points. We codify this useful fact in the below lemma:

\begin{lemma}
Suppose that there exists $\y\in\mathbf{S}(\G)$ such that $\psi$ admits at least one decomposition $\psi = r_1*r_2$ where $r_1\in\Rect^*(\x, \y)$ and $r_2\in\Rect^*(\y, \z).$ Then, $|\x - \x\cap \z| = 0, 3,$ or $4.$
\end{lemma}

\begin{lemma}\label{rectangles4}
Let $|\x - \x\cap \z| = 4.$ Suppose that there exists $\y\in\mathbf{S}(\G)$ such that $\psi$ admits at least one decomposition $\psi = r_1*r_2$ where $r_1\in\Rect^*(\x, \y)$ and $r_2\in\Rect^*(\y, \z).$ Suppose that $r_1, r_2$ are not both long. Then, $\psi$ admits precisely two decompositions $\psi = r_1*r_2 = r_1'*r_2',$ such that there exists $\y'\in\mathbf{S}(\G)$ with $r_1\in\Rect^*(\x, \y')$ and $r_2'\in\Rect^*(\y', \z)$, and $r_1', r_2'$ are not both long. Moreover, these two decompositions have the same degree.
\end{lemma}

\begin{proof}
Lift the decomposition $\psi = r_1*r_2$ into the universal cover so that $r_1$ and $r_2$ are represented by connected polygons. Because the grid states contain precisely one point in each horizontal and vertical circle, we must have that the circles containing the edges of $r_1$ and $r_2$ are all different. Hence, it is clear that the only possible corners of $\x$ that can be the outgoing corners of any decomposition $\psi = r_1'*r_2'$ are the outgoing corners of $r_1$ are the outgoing corners of $r_2$. Hence, there are clearly precisely 2 decompositions of $\psi$ as a composite of two polygons $\psi = r_1*r_2 = r_1'*r_2',$ with $r_1\in\Rect^*(\x, \y)$ and $r_1'^*\in\Rect(\x, \y').$ Furthermore, $r_1$ and $r_2'$ share the same support, as do $r_1'$ and $r_2.$ Hence, the $\mathcal{T}$ terms of the degrees of both decompositions agree. Furthermore, $r_1'$ and $r_2'$ are clearly not both long.

Let $k = \deg(r_1, r_2).$ We wish to show $\deg(r_1', r_2') = k.$ For $i=1,2,$ let $C(r_i)$ be the number of corners of $r_{3-i}$ discounting points of $\beta_i\cap\gamma_i$ intersecting $\Int(r_i)$; suppose without loss of generality that $C(r_1)\geq C(r_2).$ Then, there are four cases here: the first is that $C(r_1)=0.$ In this case, clearly $$|\Int(r_1')\cap \x|+|\Int(r_2')\cap \y'| = |\Int(r_1)\cap \x|+|\Int(r_2')\cap \y|$$ since $\Int(\psi)\cap\x = \Int(\psi)\cap\y = \Int(\psi)\cap\y'$ (recall that $\x, \y$ and $\x, \y'$ are only different in 2 places.) The remaining cases have $C(r_1)>0,$ so $r_1$ is not thin, therefore not long. The second case is that $C(r_1)=1$; in this case $$|\Int(r_1')\cap \x|+|\Int(r_2')\cap \y'| = |\Int(r_1)\cap \x|+|\Int(r_2')\cap \y| = 1 + \psi\cap(\x\cap\z)$$ where $\psi\cap(\x\cap\z)$ is counted with multiplicity. The third case is that $\Int(r_1)$ contains exactly 2 corners of $r_2$; again in this case $$|\Int(r_1')\cap \x|+|\Int(r_2')\cap \y'| = |\Int(r_1)\cap \x|+|\Int(r_2')\cap \y| = 1 + \psi\cap(\x\cap\z)$$ where $\psi\cap(\x\cap\z)$ is counted with multiplicity. Finally, we could have that $\Int(r_1)$ contains all 4 corners of $r_2$; in this case $$|\Int(r_1')\cap \x|+|\Int(r_2')\cap \y'| = |\Int(r_1)\cap \x|+|\Int(r_2')\cap \y| = 2 + \psi\cap(\x\cap\z).$$ That suffices for the proof.
\end{proof}

\begin{lemma}\label{rectangles3}
Let $|\x - \x\cap \z| = 3.$ Suppose that in the support of $\psi,$ an entire row or column may not have multiplicity $\geq2,$ and that at most one entire row or at most one entire column may have multiplicity 1. (These conditions are achieved, for instance, when only 1 of the two rectangles in the decomposition is allowed to be long.)

Suppose that there exists $\y\in\mathbf{S}(\G)$ such that $\psi$ admits at least one decomposition $\psi = r_1*r_2$ where $r_1\in\Rect^*(\x, \y)$ and $r_2\in\Rect^*(\y, \z).$ Suppose that $r_1, r_2$ are not both long. Then, $\psi$ admits precisely two decompositions $\psi = r_1*r_2 = r_1'*r_2',$ such that there exists $\y'\in\mathbf{S}(\G)$ with $r_1\in\Rect^*(\x, \y')$ and $r_2'\in\Rect^*(\y', \z)$, and $r_1', r_2'$ are not both long. Moreover, these two decompositions have the same degree.
\end{lemma}

\begin{proof}
Consider a lift of $\psi$ to the universal cover of the torus such that $\psi$ is represented by a connected L-shaped polygon $Q$. Then, $\psi = r_1*r_2,$ where $r_1$ and $r_2$ are represented by rectangles in the universal cover with disjoint interiors (which may not be disjoint when we project back down to the torus.)

Since $|\x - (\x\cap\z)| = 3$, the two rectangles $r_1$ and $r_2$ must share a corner $c.$ Since this corner must be incoming for $r_1$ and outgoing for $r_2,$ then the two rectangles must create a 180-degree angle at this corner, and hence their intersection $r_1\cap r_2$ is an edge $e.$ The boundary $\partial e$ is thus two points, $c$ and another point, which we shall call $d.$ Clearly, there exists a 270-degree angle at $d$. In any decomposition of $\psi,$ there cannot be a 270-degree angle. Since there are precisely two ways to cut $Q$ at this angle, and each one uniquely specifies a decomposition, then we get $\psi$ has precisely two decompositions $\psi = r_1*r_2 = r_1'*r_2',$ and $r_1', r_2'$ are not both long by the conditions on the support of $\psi.$

We must show these two decompositions have the same degree. First, note that any point of $\x$, $\y$, or $\y'$ inside of $\Int(Q)$ must not lie on $\partial(r_1)\cup\partial(r_2),$ since this would contradict the fact that grid states contain only 1 point on each horizontal or vertical circle. If $Q$ can embed into a fundamental domain of the torus, then $\mathcal{T}=0$ for all rectangles in all decompositions, and the local multiplicities are $\leq1.$ Thus, $\x\cap\Int(Q)=\y\cap\Int(Q)=\y'\cap\Int(Q),$ so the degrees are the same.

Suppose $Q$ cannot embed into a fundamental domain. Then, by the multiplicity constraints, we must have that one of the two decompositions involves a long rectangle $t\in\Long(\x, \y)$ or $\Long(\y, \z)$ for some grid state $\y,$ and some other rectangle $r\in\Rect(\x, \y)$ or $\Rect(\y, \z)$ such that $r\cup t$ embeds in a fundamental domain if we delete the annulus contained by $t.$ Suppose the intermediate stage in the other decomposition, $r_1'*r_2',$ is $\y'.$ If one of $r_1'$ or $r_2'$ is long, then clearly no corner of one can lie in the interior of the other, hence the degrees are clearly the same. 

Otherwise, $r_1', r_2'\in\Rect(\x,\y')\cup\Rect(\y', \z).$ We see geometrically that $\Int(r_1')\cap\x$ must contain precisely one point $c$ of $\x$ (a corner of $r_2'$) that is not contained in $\Int(r)$ or $\Int(t),$ since it lies on the edges of both such rectangles. Since $t$, and hence $r_2',$ is thin, all points of $\Int(Q)\cap\x-c = \Int(Q)\cap\y=\Int(Q)\cap\y'$ must lie in $\Int(r)$ and $\Int(r_1).$ The degree contribution of the point $c$ in the decomposition $r_1'*r_2'$ is exactly canceled out by the contribution of $\mathcal{T}(t)$ in the decomposition $r*t.$ That suffices for the proof.
\end{proof}

\begin{lemma}\label{rectangles0}
Let $|\x - \x\cap \z| = 0$ (so $\x=\z.$) Suppose that in the support of $\psi,$ an entire row or column may not have multiplicity $\geq2,$ and that at most one entire row or at most one entire column may have multiplicity 1.

Suppose that there exists $\y\in\mathbf{S}(\G)$ such that $\psi$ admits at least one decomposition $\psi = r_1*r_2$ where $r_1\in\Rect^*(\x, \y)$ and $r_2\in\Rect^*(\y, \x).$ Then, this decomposition is unique and $\psi$ is an annulus (either horizontal or vertical) of width 1 (such that the multiplicity of $\psi$ in each square is $\leq1.$)
\end{lemma}

\begin{proof}
Lift $\psi$ to a connected polygon $Q$ in the universal cover, such that $r_1$ and $r_2$ are represented by rectangles $R_1, R_2$ in the universal cover projecting onto $r_1, r_2$ in the torus. Then, the condition that each grid state must only contain one point in each row or column forces $R_1$ and $R_2$ to share an edge. Hence, $Q$ is a rectangle; since all the corners of $Q$ must be members of $\x,$ the condition that each grid state must only contain one point in each row or column forces $Q$ to be an annulus. It must have multiplicity 1 and have width one by the conditions on the support and multiplicities of $\psi.$

Furthermore, such annulus has a unique decomposition since the first rectangle $r_1$ must have outgoing corners precisely $\x\cap\text{clo}(Q),$ which is two points; since $r_1$ is not long, this determines $r_1$ uniquely. This in turn determines $r_2$ uniquely.
\end{proof}

\subsection{Pentagon and Hexagon Decomposition Lemmas}

We now suppose that $\G$ and $\G'$ are two grid diagrams which have been superimposed as in the previous section. Let $\x\in \mathbf{S}(\G)$ and $\z\in \mathbf{S}(\G')$ or $\mathbf{S}(\G)$, and $\psi\in\pi(\x, \z).$

\begin{remark}
A very important warning is that for this entire subsection, we assume that our regions have empty intersection with $\X.$ By Remark \ref{bigon}, this means that our regions may never contain an entire bigon in their support. This is not strictly necessary for most of the proofs below, but speeds up the arguments nicely.
\end{remark}

\begin{definition}
The \emph{closest point map} $I:\mathbf{S}(\G)\rightarrow\mathbf{S}(\G')$ is defined by letting $I(\x)$ be the grid state in $\G'$ which matches $\x$ in all but one point: the point $\alpha_j\cap\beta_i\in\x$ is replaced by the point $\alpha_j\cap\gamma_i\in I(\x).$
\end{definition}

We record here the following useful lemma:

\begin{lemma}\label{pentagon_grading}
For a grid state $\x \in \mathbf{S}(\G),$ we have that $M(\x) - M(I(\x)) = -1 + 2|t\cap\mathbb{O}|,$ where $t$ is the unique triangular subset of one of the bigons bounded by $\beta_i$ and $\gamma_i$ with two corners in common with $(I(\x) - \x)\cup(\x - I(\x)).$
\end{lemma}

\begin{proof}
This is demonstrated in the proof of [OSS] Lemma 5.1.3.
\end{proof}

\begin{definition}
Consider a (possibly long) pentagon or hexagon $\psi\in\pi(\x, \y)$, where $\x\in\mathbf{S}(\G)$ and $\y\in\mathbf{S}(\G)\cup\mathbf{S}(\G').$ By slight abuse of notation, let $I(\y)\in\mathbf{S}(\G)$ be $\y$ if $\y\in\mathbf{S}(\G)$ and $I(\y)$ if $\y\in\mathbf{S}(\G')$.

Then, the associated \emph{associated rectangular domain} $\Psi\in\Rect^*(\x,I(\y))$ of $\psi$ is the (possibly long) rectangle from $\x$ to $I(\y)$ whose multiplicities are identical to those of $\psi$ outside the bigons between $\beta_i$ and $\gamma_i.$

Denote by $R(p)$ the associated rectangular domain of a (possibly long) pentagon or hexagon $p.$
\end{definition}

Note that $\Psi$ is long if and only if $\psi$ is long, and that $\mathcal{T}(\Psi)=\mathcal{T}(\psi).$

The following lemma is immediate from the above definition:

\begin{lemma}
Let $\psi \in \pi(\x, \z)$ and suppose that $\psi\cap\X=\emptyset.$ A decomposition $\psi = p_1*p_2,$ where $p_1, p_2$ are either rectangles, pentagons, or hexagons, and not both long, corresponds to a decomposition of the associated rectangular domain $\Psi$ into two rectangles, not both long. Furthermore, the decompositions of $\psi$ and $\Psi$ have the same degree.
\end{lemma}

\begin{corollary}\label{pents}
Let $\psi \in \pi(\x, \z)$ and suppose that $\psi\cap\X=\emptyset.$ Suppose $|\x - \x\cap\z| = 3$ or $4$. Suppose $\psi$ admits a decomposition $\psi = p_1*p_2,$ where $p_1, p_2$ are either rectangles, pentagons, or hexagons, and not both long. Then $\psi$ admits two decompositions as such, and both have the same degree. Furthermore, if $\psi = p*r$ or $r*p$ where $p$ is a long pentagon and $r$ is a (not long) rectangle, then the other decomposition of $\psi$ is also as a long pentagon and a not long rectangle.
\end{corollary}

\begin{proof}
If $|\x - \x\cap\z| = 4,$ then only one of $p_1$ or $p_2$ can have an edge on the curves $\beta_i$ or $\gamma_i.$ Furthermore, the associated rectangular domain $\Psi$ also connects two grid states differing by 4 points. Hence, Lemma \ref{rectangles4} tells us $\Psi$ admits two rectangular decompositions of the same degree. Since at most one of $p_1$ or $p_2$ can have an edge on the curves $\beta_i$ or $\gamma_i,$ clearly each of the two decompositions of $\Psi$ corresponds uniquely to a decomposition of $\psi.$

If $|\x - \x\cap\z| = 3,$ the associated rectangular domain $\Psi$ also connects two grid states differing by 3 points. Hence, Lemma \ref{rectangles4} tells us $\Psi$ admits two rectangular decompositions of the same degree. There are two cases: either, a point $a$ or $b$ at which $p_1$ or $p_2$ has a corner is on an edge shared by both rectangles in one of the decompositions of $\Psi$, or it is not. In the latter case, each of the two decompositions of $\Psi$ corresponds uniquely to a decomposition of $\psi.$ In the former case, one of the two decompositions of $\Psi$ does not correspond to a decomposition of $\psi,$ however we achieve precisely one more decomposition of $\psi$ by removing a portion of one of the bigons from the support of one of $p_1$ or $p_2$ and appending it to the other, which is clearly possible since the support of a pentagon or hexagon cannot contain an entire bigon.

The last claim in the above corollary follows simply because any long rectangle in a decomposition of $\psi$ would clearly have to intersect $\X$ (here is one instance where assuming $\psi\cap\X=\emptyset$ drastically simplifies our argument.)
\end{proof}

\section{Sign Assignments}

The goal of this section is to define double-point grid homology over the integers, and verify that this is indeed a knot invariant.

Fix a grid diagram $\G,$ and let $\Rect(\G) = \bigcup_{\x, \y\in\mathbf{S}(\G)} \Rect(\x, \y),$ and also $\Rect^*(\G) = \bigcup_{\x, \y\in\mathbf{S}(\G)} \Rect^*(\x, \y).$ 

\begin{definition}
(Modified from [OSS Definition 15.1.2]). A \emph{sign assignment} is a function $S:\Rect(\G)\rightarrow\{-1, 1\}$ satisfying three properties: (1) if there exists a domain $\psi$ such that $\psi = r_1*r_2 = r_1'*r_2'$ for $r_1, r_2, r_1', r_2'\in\Rect(\G),$ then $$S(r_1)S(r_2) = - S(r_1')S(r_2'),$$ (2) if $r_1*r_2$ is a horizontal annulus, then $S(r_1)S(r_2)=1,$ and (3) if $r_1*r_2$ is a vertical annulus, then $S(r_1)S(r_2)=-1.$

An \emph{extended sign assignment} is a function $S:\Rect^*(\G)\rightarrow\{-1, 1\}$ satisfying properties (2) and (3) above, and also (1'): if there exists a domain $\psi$ such that $\psi = r_1*r_2 = r_1'*r_2'$ for $r_1, r_2, r_1', r_2'\in\Rect^*(\G),$ then $$S(r_1)S(r_2) = - S(r_1')S(r_2').$$
\end{definition}

\begin{definition}
For $\psi\in\pi(\x, \z),$ let $D(\psi)$ be the set of decompositions of $\psi$ into two rectangles, and $D^*(\psi)$ the set of decompositions of $\psi$ into one rectangle and one long rectangle (in either order). For $d=r_1*r_2\in D(\psi)\cup D^*(\psi),$ define the sign $S(d)$ to be $S(r_1)*S(r_2).$
\end{definition}

\begin{remark}
Note that if $r_1$ and $r_2$ are two rectangles with $r_1\in\Rect(\x, \y)$ and $r_2\in\Rect(\z, \w),$ it is possible that $r_1$ and $r_2$ have the same support but $S(r_1)\neq S(r_2).$ It is therefore important to emphasize that $S$ is a function of the domain, initial, and final grid states.
\end{remark}

Note further that by the definition of our sign assignments, the proofs of lemmas \ref{rectangles4} and \ref{rectangles3} extend to prove the following slightly stronger statement:

\begin{lemma}\label{signs}
Under the assumptions of lemma \ref{rectangles4} or \ref{rectangles3}, the two decompositions of $\psi$ resulting from those lemmas have opposite signs.
\end{lemma}

Given a sign assignment, we may define a new chain complex over the integers. Let $GCL(\G; \mathbb{Z})$ be the free $\mathbb{Z}[V_1,\dots, V_n, v]$-module generated by the grid states of $\G,$ with mutliplication by each $V_i$ homogeneous of bigrading $(-2, -1),$ and multiplication by $v$ homogeneous of bigrading $(2, 0).$ Define the differential as follows on grid states, extending by linearity: $$\partial_S\x= \sum_{\y\in\mathbf{S}(\G)}\sum_{\{r\in\Rect(\x, \y)\big| r\cap\X = \emptyset\}} S(r)v^{\mathcal{T}(r)}V_1^{O_1(r)}\dots V_n^{O_n(r)}\y.$$

\begin{proposition}\label{diff}
$\partial_{S}^2=0$ for any sign assignment $S.$
\end{proposition}

\begin{proof}
By expanding out definitions, we see $$\partial_S^2\x = \sum_{\z\in\mathbf{S}(\G)}\sum_{\{\psi\in\pi(\x, \z)\big| \psi\cap\X = \emptyset\}} \sum_{d\in D(\psi)}S(d)v^{\deg(d)}V_1^{O_1(\psi)}\dots V_n^{O_n(\psi)}\y.$$

Hence, it is sufficient to show that domains $\psi\in\pi(\x, \z)$ with $|D(\psi)|>0$ have exactly two decompositions $d\in D(\psi)$ with equal degree and opposite signs.

If $|\x -\x\cap\z| = 4$ or $|\x -\x\cap\z| = 3,$ then Lemma \ref{signs} tells us immediately that any $\psi\in\pi(\x, \z)$ with $|D(\psi)|>0$ has exactly two decompositions $d\in D(\psi)$ with equal degree and opposite signs.

The last case is $\x = \z.$ Lemma \ref{rectangles0} tells us that in this case, $\psi$ is an annulus, hence $\psi\cap\X\neq\emptyset,$ so this also contributes 0 to the equation.
\end{proof}

Let $GHL_S(\G; \mathbb{Z})$ be the homology of this chain complex. We will eventually see that the $V_i$ are homotopic to each other, hence calling the induced multiplication $U,$ we get $GHL_S(\G; \mathbb{Z})$ is a bigraded $\mathbb{Z}[U, v]$-module.

We wish to prove the following theorem, which is an analog of the invariance of ordinary grid homology over the integers:

\begin{theorem}\label{integer}
For each grid diagram, there exists a sign assignment; furthermore, all sign assignments produce isomorphic homology $GHL_S(\G; \mathbb{Z}).$ This homology $GHL_S(\G; \mathbb{Z}),$ as a bigraded $\mathbb{Z}[U, v]$-module, is also a knot invariant.
\end{theorem}

We divide the proof of this theorem into several steps. First, we lift existence and uniqueness of sign assignments for grid homology to our situation, which is rather simple:

\begin{lemma}
For each grid diagram, there exists a sign assignment; furthermore, all sign assignments produce isomorphic homology $GHL_S(\G; \mathbb{Z}).$
\end{lemma}

\begin{proof}
The existence of a sign assignment follows immediately from [OSS] Theorem 15.1.5. Their proof of this theorem also shows that for any two sign assignments $S_1$ and $S_2,$ there exists a function $g:\mathbf{S}(\G)\rightarrow \{-1, 1\}$ such that $S_2(r) = g(\x)S_1(r)g(\y)$ for each $r\in\Rect(\x, \y).$ Hence, we may define a $\mathbb{Z}[V_1,\dots, V_n, v]$-module homomorphism from $(GCL(\G; \mathbb{Z}), \partial_{S_1})$ to $(GCL(\G; \mathbb{Z}), \partial_{S_2})$ by $\x\mapsto g(\x)\x.$ It is immediate that this is a chain map and an isomorphism of $\mathbb{Z}[V_1,\dots, V_n, v]$-modules, and this suffices for the proof. (This is an analog of the map defined in [OSS] Proposition 15.1.10.)
\end{proof}

Next, we show that the $V_i$ are homotopic. For this, we require a lemma about extended sign assignments, and some definitions. For any long rectangle $r\in\Rect^*(\x, \y),$ deleting an annulus leaves a rectangle $r'\in\Rect^*(\x,\y)$ which we will call the \emph{associated short rectangle.}

\begin{lemma}
Each sign assignment $S:\Rect(\G)\rightarrow\{-1, 1\}$ may be extended to an extended sign assignment, such that if $r\in\Rect^*(\x,\y)$ is long and $r'\in\Rect^*(\x,\y)$ is its associated short rectangle then $S(r) = S(r').$
\end{lemma}

\begin{proof}
We use the notation of [OSS Chapter 15]. Let $\tilde{\mathfrak{S}}_n$ denote the spin extension of the symmetric group on $n$ letters (see [OSS Section 15.2]), and let $\tilde{T}$ denote the set of lifts of transpositions in $\mathfrak{S}_n$; these are indexed by ordered pairs of distinct integers in $\{1,\dots, n\}$ and denoted by $\tilde{\tau}_{i,j}$ for $i\neq j.$ Here, $\tilde{\tau}_{i,j}$ and $\tilde{\tau}_{j,i}$ are the two lifts of the permutation $(i j)\in\mathfrak{S}_n.$ Also, $z\neq1\in\tilde{\mathfrak{S}}_n$ is the other lift of the identity $1\in\mathfrak{S}_n.$

Let $\tilde{\tau}:\Rect^*(\G)\rightarrow \tilde{T}$ be the map sending $r\in\Rect^*(\G)$ to $\tilde{\tau}_{i,j}$ where the southwest corner of the associated short rectangle $r'$ is on $\beta_i$ and the northeast corner of $r$ is on $\beta_i$. We wish to show the following three conditions hold:
\begin{itemize}
	\item If there is a region $\psi$ with two decompositions $\psi = r_1*r_2 = r_1'*r_2'$ for $r_1, r_2, r_1', r_2'\in\Rect^*(\G)$ with at most one of $r_1, r_2$ long, then $$\tilde{\tau}(r_1)\cdot\tilde{\tau}(r_2) = z \cdot\tilde{\tau}(r_1')\cdot\tilde{\tau}(r_2').$$
	\item If $r_1*r_2$ forms a horizontal annulus of multiplicity 1, then: $$\tilde{\tau}(r_1)\cdot\tilde{\tau}(r_2) = 1.$$
	\item If $r_1*r_2$ forms a vertical annulus of multiplicity 1, then: $$\tilde{\tau}(r_1)\cdot\tilde{\tau}(r_2) = z.$$
\end{itemize}

Assuming for now that these three conditions hold, we will demonstrate the existence of the extension of our sign assignment $S.$ For each grid state $\x,$ let $\sigma_{\x}\in\mathfrak{S}_n$ be the corresponding permutation which is determined uniquely by $\alpha_i\cap\beta_{\sigma_{\x}(i)}\in\x.$ [OSS Section 15.2] proves that every sign assignment, in particular $S,$ is given, for $r\in\Rect(\x, \y)$, as $S(r) = \tilde{\tau}(r)^{-1}\gamma(\sigma_{\x}^{-1})\gamma(\sigma_{\y}),$ for some section of the spin extension $\gamma:\mathfrak{S}_n\rightarrow\tilde{\mathfrak{S}}_n.$ Furthermore, the proof of [OSS Proposition 15.2.12] tells us that if $S$ is extended to long rectangles by the same formula, $S(r) = \tilde{\tau}(r)^{-1}\gamma(\sigma_{\x}^{-1})\gamma(\sigma_{\y}),$ then $S$ will satisfy the conditions of an extended sign assignment if $\tilde{\tau}$ obeys the three above conditions.

Thus, it is sufficient to show that $\tilde{\tau}$ obeys the three above conditions. The second and third are immediate from [OSS Section 15.2] since they do not apply to long rectangles. The first condition is immediate in the case when both decompositions of $\psi$ involve one long rectangle, since deleting an annulus from $\psi$ and in turn from each of the long rectangles does not change the map $\tilde{\tau}$ and this condition now follows from the equivalent condition for not long rectangles.

Finally, we must show this condition is true when one decomposition, $r_1*r_2,$ of $\psi$ involves a long rectangle and the other, $r_1'*r_2',$ does not. Say that $r_1$ is long and $r_1'$ has width one. This reduces to eight cases, corresponding to whether $r_1$ horizontal or vertical, and whether the multiplicity 2 portion of $\psi$ is in the Southwest, Southeast, Northwest, or Northeast corner of $r_2'.$ For each, it is a simple computation involving the relations among the members of $\tilde{T}.$
\end{proof}

\begin{definition}
We define the sign-refined homotopy operator as follows. Let $X_i\in\X$ share a row with $O_i$. Then, define the function: $$\mathcal{H}_{i, S}(\x) = \sum_{\y\in\mathbf{S}(\G)}\sum_{\{r\in\Rect^*(\x, \y)\big| r\cap\X = X_i, X_i(r)=1\}} S(r)v^{\mathcal{T}(r)}V_1^{O_1(r)}\dots V_n^{O_n(r)}\y.$$
\end{definition}

\begin{proposition}
Suppose $X_i$ shares a column with $O_j.$ Then following equation holds: $$\mathcal{H}_{i, S}\circ\partial_S + \partial_S\circ\mathcal{H}_{i, S} = V_i - V_j.$$
\end{proposition}

\begin{proof}
By expanding out definitions, we see the left-hand side is given as $$\sum_{\z\in\mathbf{S}(\G)}\sum_{\{\psi\in\pi(\x, \z)\big| \psi\cap\X = X_i, X_i(\psi)=1\}} \sum_{d\in D^*(\psi)}S(d)v^{\deg(d)}V_1^{O_1(\psi)}\dots V_n^{O_n(\psi)}\y.$$

If $|\x -\x\cap\z| = 4$ or $|\x -\x\cap\z| = 3,$ then Lemma \ref{signs} tells us immediately that any $\psi\in\pi(\x, \z)$ with $|D(\psi)|>0$ has exactly two decompositions $d\in D(\psi)$ with equal degree and opposite signs. Hence, all such $\psi$ contribute zero to the above sum.

The last case is $\x = \z.$ Lemma \ref{rectangles0} tells us that in this case, $\psi$ is an annulus, hence $\psi\cap\X=X_i$ has precisely two solutions, a horizontal thin annulus and a vertical thin annulus. (The multiplicities of these annuli must be 1 since $X_i(\psi)=1.$) The horizontal thin annulus contributes $V_i$ and the vertical thin annulus contributes $-V_j.$
\end{proof}

\begin{proposition}
For any $k, l,$ we have $V_k$ and $V_l$ are chain homotopic.
\end{proposition}

\begin{proof}
The above proposition shows that $V_i$ and $V_j$ are chain homotopic whenever $O_i$ and $O_j$ are consecutive $O$-markings as we traverse the knot. Since the knot has one component, each $O_k$ and $O_l$ are both members of a finite sequence of consecutive $O$-markings, which suffices for the proof.
\end{proof}

The remainder of the proof of invariance of the sign-refined homology $GHL_S(\G;\mathbb{Z})$ is to show it is a knot invariant. By Theorem \ref{cromwell}, it suffices to show that this homology is unchanged under commutations and switches, and also under (de)-stabilizations of type $X:SW.$

\subsection{Commutation and Switch Invariance}

First, we show this for the commutations and switches; we model our arguments off of those in [OSS Section 15.3]. As in Section 2.1, we take two grid diagrams $\G$ and $\G'$ related by a commutation or switch and superimpose them; we borrow here the notation from that section. We will do the computations below for a column commutation or switch; a row commutation or switch proceeds identically. Very concretely, the goal of this section is to prove the following theorem:

\begin{theorem}
If two grid diagrams $\G$ and $\G'$ as above differ by a column commutation or switch, then $GCL_S(\G;\mathbb{Z})$ and $GCL_S(\G';\mathbb{Z})$ are quasi-isomorphic chain complexes.
\end{theorem}

Again, we recall that by Remark \ref{bigon}, we may guarantee that each bigon bounded by $\beta_i$ and $\gamma_i$ contains at least one $\X$-marking inside.

First, we must define pentagon and hexagon maps; the pentagon maps will provide the quasi-isomorphism and the hexagon map will be the relevant homotopy operator. For this, we need signs for pentagons and hexagons.

\begin{definition}
For a pentagon $p\in\Pent^*(\x, \y'),$ let $P$ be its associated rectangular domain. We define the sign of $p$ as follows: $$S(p) := (-1)^{M(\x)+B(p)}S(P),$$ where $B(p)$ is 1 if $p$ lies to the left of $\beta_i$ and 0 if $p$ lies to the right of $\beta_i.$ Similarly, for $p\in\Pent^*(\x', \y),$ let $P$ be its associated rectangular domain. We define the sign of $p$ as follows: $$S(p) := (-1)^{M(\y)+B(p)}S(P).$$

For a hexagon $h\in\Hex^*(\x, \y),$ let $H$ be its associated rectangular domain. We define the sign of $h$ simply as $S(h) := S(H).$
\end{definition}

\begin{definition}
The pentagon map $P_S:GCL_S(\G;\mathbb{Z})\rightarrow GCL_S(\G';\mathbb{Z})$ is defined as $$P_S(\x) = \sum_{\y'\in\mathbf{S}(\G')}\sum_{\{p\in\Pent^*(\x, \y')\big| p\cap\X = \emptyset\}} S(p)v^{\mathcal{T}(p)}V_1^{O_1(p)}\dots V_n^{O_n(p)}\y',$$ and we define a similar map $P_S':GCL_S(\G';\mathbb{Z})\rightarrow GCL_S(\G;\mathbb{Z})$ by $$P_S'(\x') = \sum_{\y\in\mathbf{S}(\G)}\sum_{\{p\in\Pent^*(\x', \y)\big| p\cap\X = \emptyset\}} S(p)v^{\mathcal{T}(p)}V_1^{O_1(p)}\dots V_n^{O_n(p)}\y.$$ The hexagon map $H_S:GCL_S(\G;\mathbb{Z})\rightarrow GCL_S(\G;\mathbb{Z})$ is defined as $$H_S(\x) = \sum_{\y\in\mathbf{S}(\G)}\sum_{\{h\in\Hex^*(\x, \y)\big| h\cap\X = \emptyset\}} S(h)v^{\mathcal{T}(h)}V_1^{O_1(h)}\dots V_n^{O_n(h)}\y.$$
\end{definition}

\begin{proposition}
The pentagon map $P_S$ is a bigraded quasi-isomorphism.
\end{proposition}

\begin{proof}
First, we show $P_S$ is bigraded. This is an immediate consequence of the relative formulas for Maslov and Alexander gradings as well as Lemma \ref{pentagon_grading}.

Next, we show $P_S$ is a chain map; this amounts to showing that $P_S\circ\partial_S-\partial_S\circ P_S = 0.$ 
For a domain $\psi\in\pi(\x, \z'),$ let $D_{pr}(\psi)$ be the set of all decompositions of $\psi$ as a (possibly long) pentagon and a (not long) rectangle in that order, and $D_{rp}(\psi)$ the reverse.
Expanding out the equation, we wish to show $$\sum_{\z'\in\mathbf{S}(\G')}\sum_{\{\psi\in\pi(\x, \z')\big| \psi\cap\X = \emptyset\}} \sum_{d\in D_{pr}(\psi)} S(d)v^{\deg(d)}V_1^{O_1(\psi)}\dots V_n^{O_n(\psi)}\z' $$$$= \sum_{\z'\in\mathbf{S}(\G')}\sum_{\{\psi\in\pi(\x, \z')\big| \psi\cap\X = \emptyset\}} \sum_{d\in D_{rp}(\psi)} S(d) v^{\deg(d)}V_1^{O_1(\psi)}\dots V_n^{O_n(\psi)}\z'.$$

In this case, we see immediately that $|\x - \x\cap\z'|$ could be 3, 4, or 1. If it is 4, then the proof of Corollary  \ref{pents} tells us that any $\psi$ with $|D_{rp}(\psi)\cup D_{pr}(\psi)|\geq1$ has exactly two decompositions $r*p$ and $p'*r'$ of equal degree corresponding to the two rectangular decompositions of $\Psi.$ To show the contribution of $\psi$ is equal on both sides of the equation, it is sufficient to show that the signs of these two decompositions is equal. By the definition of a sign assignment, and the decompositions of $\Psi,$ we see that $$S(r)S(R(p)) = -S(R(p'))S(r').$$ Since the pentagons $p$ and $p'$ have initial grid states of opposite parity and clearly $B(p)=B(p'),$ as they have the same support, the signs $S(r*p)$ and $S(p'*r')$ are indeed equal.

Next, suppose $|\x - \x\cap\z'|=3.$ Then, as in the proof of Corollary \ref{pents}, any $\psi$ with $|D_{rp}(\psi)\cup D_{pr}(\psi)|\geq1$ has exactly two decompositions of equal degree corresponding to the two rectangular decompositions of $\Psi.$ 

The proof of Corollary \ref{pents} gives us two cases. First, that these two decompositions correspond to different decompositions of $\Psi.$ Here, the two associated rectangular decompositions have opposite signs. However, the $B$-value of the pentagons in each decomposition is obviously equal, and the Maslov initial grid states of the pentagons in each decomposition differ in parity if and only if one decomposition is in $D_{rp}(\psi)$ and the other is in $D_{pr}(\psi).$ Hence, $\psi$ contributes equally to both sides of the above equation in this case. 

Otherwise, these two decompositions correspond to the same rectangular decomposition of $\Psi,$ in which case the decompositions are of the form $r*p$ and $p'*r'$, and $b(p)=-b(p'),$ so $\psi$ also contributes equally to both sides of the above equation in this case.

The final case is when $|\x - \x\cap\z'|=1.$ Then, the associated rectangular domain is an annulus, and must therefore be thin to avoid intersection with $\X$; furthermore, the condition that each bigon contains an $X$-marking prevents this annulus from having multiplicity 2. Hence, the degrees of all decompositions must be 0, and the remainder of the proof proceeds exactly as in [OSS Lemma 15.3.3].

By identical reasoning, $P_S'$ is also a bigraded chain map.

Finally, we show a homotopy formula. Specifically, we will show that $$H_S\circ\partial_S+\partial_S\circ H_S + P_S\circ P_S' = -\text{Id},$$ hence $P_S\circ (-P_S')$ is chain-homotopic to the identity. By an analogous argument, we may also show that $(-P_S')\circ P_S$, which would complete the proof that $P_S$ is a quasi-isomorphism.

For this formula, we consider a domain $\psi\in\pi(\x, \z)$ contributing at least one nonzero term to the left-hand side. By the proof of Corollary \ref{pents}, if $|\x - \x\cap\z| =$ 3 or 4, then there are exactly two decompositions of $\psi$. If the associated rectangular decompositions of $\Psi$ are different, then we have two possibilities. Either both decompositions correspond to hexagons and rectangles only, in which case $\psi$ contributes zero to the left-hand side by the definition of signs for hexagons. Otherwise, one decomposition consists of two pentagons, with equal $B$-values, and the relevant Maslov gradings have the same parity. Hence, the two decompositions cancel in this sum. 

If the associated rectangular decompositions of $\Psi$ are the same, then one of the decompositions consists of two pentagons with different $B$-values, and the other consists of a hexagon and rectangle. Hence, the two decompositions cancel in this sum. 

The final possibility is $\x = \z$; as before the condition that $\psi\cap\X=\emptyset$ forces $\Psi$ to be a thin annulus of multiplicity 1 and width 1. Furthermore, geometrically we see it is vertical. Hence, the degrees of all the decompositions must be 0, and we are reduced to the case proven in [OSS Lemma 15.3.4], which is that $\psi$ is unique and contributes the identity to the sum. That proves the theorem, and thus the commutation and switch invariance of $GHL_S.$

\end{proof}

\subsection{Stabilization Invariance}

We now let $\G$ and $\G'$ be grid diagrams with $\G'$ a stabilization of $\G$ of type $X:SW$ (recall that by our generalization of Cromwell's Theorem, it is sufficient to consider this case.) Let $c\in \G'$ be the central point of the stabilization.

We label the 2-by-2 region introduced in the stabilization as follows: $\begin{matrix} X_1 & O_1 \\ & X_2 \end{matrix},$ and let $O_2$ be the $O$-marking sharing a row with $X_2.$

Note that $\mathbf{S}(\G') = \mathbf{I}(\G') \cup \mathbf{N}(\G') ,$ a disjoint union, where $\mathbf{I}(\G)$ are those $\x'\in\mathbf{S}(\G')$ with $c\in\x',$ and $\mathbf{N}(\G')$ are the other grid states.

There is a natural bijection $\mathbf{S}(\G)\rightarrow \mathbf{I}(\G')$ which can be thought of as $\x\mapsto \x\cup\{c\},$ and this bijection extends to a bijection of rectangles from $\x$ to $\y$ in $\mathbf{S}(\G)$ to rectangles from $\x\cup\{c\}$ to $\y\cup\{c\}$ in $\mathbf{I}(\G').$ Fix a sign assignment $S'$ on $\G',$ and define a sign assignment $S$ on $\G$ by pulling back $S'$ on $\mathbf{I}(\G').$

For a chain complex $A,$ let $A[[m, a]]$ denote $A$ with the bigradings shifted by $(+m, +a).$

Let $\mathbf{I}$ be submodule of $GCL_{S'}(\G';\mathbb{Z})$ generated by $\mathbf{I}(\G'),$ and likewise for $\mathbf{N}.$ Then, $\mathbf{I}$ is a quotient complex and $\mathbf{N}$ is a subcomplex. The bijection of $\mathbf{S}(\G)$ and $\mathbf{I}(\G')$ clearly induces an isomorphism of complexes $e: \mathbf{I}\rightarrow GCL_{S}(\G;\mathbb{Z})[[1,1]],$ since any rectangle in $\mathbf{I}$ contributing to $\partial_{S'}$ must not pass through $c,$ lest it also intersect $\X$, hence the $\mathcal{T}$-values of the rectangles are preserved in this correspondence.


Let $\pi:GCL_{S'}(\G';\mathbb{Z})\rightarrow\mathbf{I}$ denote the projection map.

 Let $C$ be the mapping cone of the map $$V_1-V_2:GCL_{S}(\G;\mathbb{Z})[V_1][[1,1]]\rightarrow GCL_{S}(\G;\mathbb{Z})[V_2].$$ Per [OSS Chapter 5], the homology of $C$ is isomorphic to that of $GCL_{S}(\G;\mathbb{Z}).$ Now, the proof of [OSS Proposition 15.3.5], which only relies on the equation $$\mathcal{H}_{2, S}\circ\partial_S + \partial_S\circ\mathcal{H}_{2, S} = V_1 - V_2$$ that we proved earlier, tells us that the map $D_S:GCL_{S'}(\G';\mathbb{Z})\rightarrow C$ given by $$D_S(\x) = (-1)^{M(\x)}(e(\x), e\circ\pi\circ \mathcal{H}_{2, S}(\x))$$ is a quasi-isomorphism.

That shows the following proposition, completing the proof of invariance of $GHL$ over the integers:

\begin{proposition}
The chain complexes $GCL_{S'}(\G';\mathbb{Z})$ and $GCL_{S}(\G;\mathbb{Z})$ have isomorphic homologies.
\end{proposition}

This theorem justifies us using the integer invariant $GHL_{S}(K;\mathbb{Z})$ in the remainder of the paper.

\section{Skein Exact Sequence}

The goal of this section is to prove a skein exact sequence for $GHL$ in analogy to the skein exact sequence satisfied by $GH^-.$ We first describe this carefully.

\subsection{Skein Exact Sequence Basics}

The idea is that if three links differ in one crossing as in the below picture (called a skein triple), then we can relate their grid homologies by an exact sequence.

\includegraphics[scale=0.4]{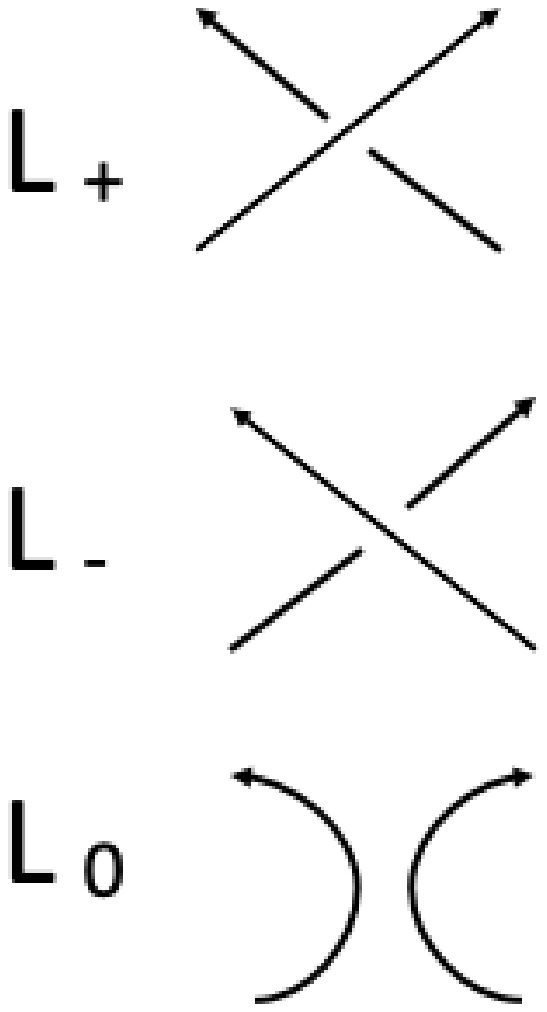}

First, we must develop a version of grid homology for links. 

\begin{definition}[Modified from \text{[OSS Definition 8.2.4]}]
Let $\G$ be a grid diagram representing an $\ell$-component link, and suppose $O_{j_1}, \dots, O_{j_\ell}$ are $O$-markings lying on each component of the link. Then, the \emph{collapsed grid complex} of $\G$ is the complex $$cGC_S^-(\G; \mathbb{Z}) := GC_S^-(\G; \mathbb{Z})/(V_{j_1}=\dots=V_{j_\ell}),$$ with the same differential. The homology of this complex is known as $cGH_S^-(\G; \mathbb{Z}).$

The \emph{collapsed double-point enhanced grid complex} of $\G$ is the complex $$cGCL_S(\G; \mathbb{Z}) := GCL_S(\G; \mathbb{Z})/(V_{j_1}=\dots=V_{j_\ell}),$$ with the same differential. The homology of this complex is known as $cGHL_S(\G; \mathbb{Z}).$
\end{definition}

The proof of the signed version of [OSS Theorem 8.2.5] adapts without variation to the double-point enhanced case to give the following theorem:

\begin{theorem}
For any link, the collapsed double-point enhanced grid complex of a grid diagram representing the link is a link invariant as a bigraded $\mathbb{Z}[U, v]$-module.
\end{theorem}

Now, we are ready to state the theorem we wish to prove. We omit the $\mathbb{Z}$ and $S$ from the notation for convenience. Let $cGHL_m(L, a)$ be the $\mathbb{Z}$-submodule of $cGHL(L)$ consisting of homogenous elements of bidegree $(m,a).$

\begin{theorem}\label{skein}
Let $(L_+, L_-, L_0)$ be an oriented skein triple, with $\ell$ and $\ell_0$ the number of components of $L_+$ and $L_0$ respectively. If $\ell_0=\ell+1,$ then there is a long exact sequence where the maps below fit together to be homomorphisms of $\mathbb{Z}[U, v]$-modules:

$$\to cGHL_m(L_+, s) \to cGHL_m(L_-, s) \to cGHL_{m-1}(L_0, s) \to cGHL_{m-1}(L_+, s) \to$$

Let $J$ be the 4-dimensional bigraded abelian group $J\cong \mathbb{Z}^4$ with one generator in bigrading $(0,1),$ one generator in bigrading $(-2, -1),$ and two generators in bigrading $(-1, 0).$

If $\ell_0=\ell-1,$ then there is a long exact sequence where the maps below fit together to be homomorphisms of $\mathbb{Z}[U, v]$-modules:

$$\to cGHL_m(L_+, s) \to cGHL_m(L_-, s) \to cGHL_{m-1}(L_0, s)\otimes J \to cGHL_{m-1}(L_+, s) \to$$
\end{theorem} 

\subsection{Proof of the Theorem}

Per [OSS Chapter 9], we may assume that $L_+, L_-,$ and $L_0$ are represented by grid diagrams $\G_+, \G_-,$ and $\G_0,$ respectively, which we picture below along with another diagram $\G_0'$ also representing $L_0.$ The below diagram is borrowed from page 153 of [OSS]. 

\includegraphics[scale=0.8]{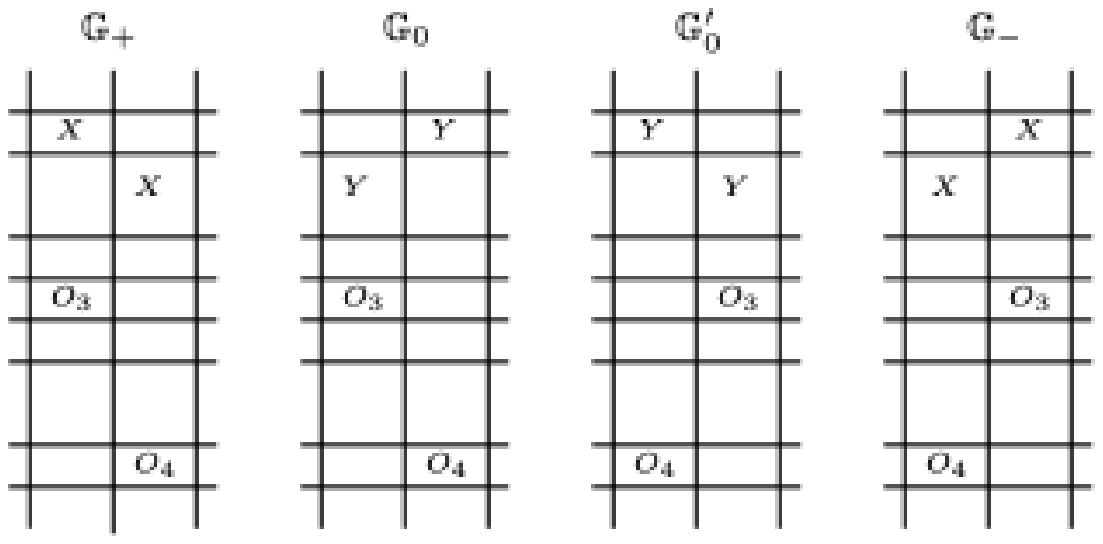}

Below, we reproduce figure 9.3 from [OSS], which depicts all four of these grid diagrams simultaneously, and defines for us two crucial points $c$ and $c'.$

\includegraphics[scale=0.8]{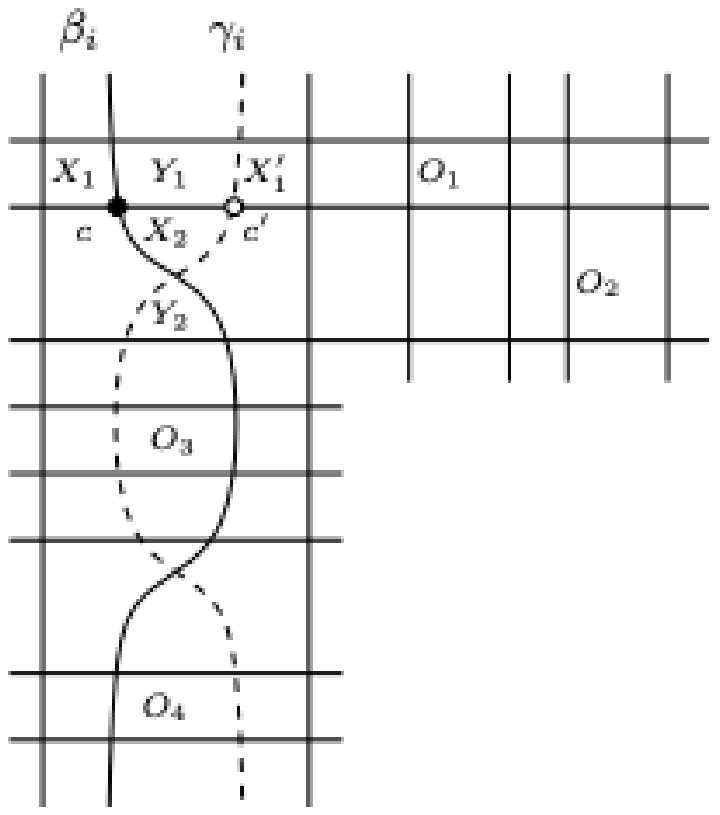}

We partition our four chain complexes, as above, into $\mathbf{I}$ and $\mathbf{N}$ parts depending on whether the grid states contain the marked point $c,$ and $\mathbf{I'}$ and $\mathbf{N'}$ parts depending on whether the grid states contain the marked point $c',$ giving the following descriptions of $GCL_S$ as mapping cones of the following maps counting some of the distinguished squares in the above diagram. (We omit the $;\mathbb{Z}$ for notational simplicity.)

\begin{tabular}{ c | c | c | c | c }
Chain Complex & Quotient Complex & Map & Subcomplex & Map Counts ... \\
\hline
\hline
\hline
$GCL_S(\G_+)$ & $(\mathbf{I}, \partial^{\mathbf{I}}_{\mathbf{I}})$ & $\partial^{\mathbf{N}}_{\mathbf{I}}:\mathbf{I} \rightarrow \mathbf{N}$ & $(\mathbf{N}, \partial^{\mathbf{N}}_{\mathbf{N}})$ & rectangles crossing $Y_1$ or $Y_2$ \\
\hline
$GCL_S(\G_0)$ & $(\mathbf{N}, \partial^{\mathbf{N}}_{\mathbf{N}})$ & $\partial^{\mathbf{I}}_{\mathbf{N}}:\mathbf{N} \rightarrow \mathbf{I}$ & $(\mathbf{I}, \partial^{\mathbf{I}}_{\mathbf{I}})$ & rectangles crossing $X_1$ or $X_2$ \\
\hline
$GCL_S(\G_0')$ & $(\mathbf{I}^{'}, \partial^{\mathbf{I}^{'}}_{\mathbf{I}^{'}})$ & $\partial^{\mathbf{N}^{'}}_{\mathbf{I}^{'}}:\mathbf{I}^{'} \rightarrow \mathbf{N}^{'}$ & $(\mathbf{N}^{'}, \partial^{\mathbf{N}^{'}}_{\mathbf{N}^{'}})$ & rectangles crossing $X_1$ or $X_2$ \\
\hline
$GCL_S(\G_-)$ & $(\mathbf{N}^{'}, \partial^{\mathbf{N}^{'}}_{\mathbf{N}^{'}})$ & $\partial^{\mathbf{I}^{'}}_{\mathbf{N}^{'}}:\mathbf{N}^{'} \rightarrow \mathbf{I}^{'}$ & $(\mathbf{I}^{'}, \partial^{\mathbf{I}^{'}}_{\mathbf{I}^{'}})$ & rectangles crossing $Y_1$ or $Y_2$ \\
\end{tabular}

\begin{definition}
We define the map $T:\mathbf{I}^{'}(\G_0')\rightarrow\mathbf{I}(\G_+)$ by the property that $T(\x) - (T(\x)\cap\beta_i) = \x - (\x\cap\gamma_i).$ (See [OSS p. 155].)
\end{definition}

\begin{lemma}
The identification $T:\mathbf{I}'(\G_0')\rightarrow \mathbf{I}(\G_0)$ extends to an isomorphism of chain complexes over $\mathbb{F}[v, V_1, \dots, V_n].$
\end{lemma}

\begin{proof}
The proof of [OSS Lemma 9.2.3] holds in our situation.
\end{proof}

\begin{lemma}
The maps $\partial^{\mathbf{I}^{'}}_{\mathbf{N}^{'}}\circ\partial^{\mathbf{N}^{'}}_{\mathbf{I}^{'}}:\mathbf{I}^{'}\rightarrow\mathbf{I}^{'}$ and $\partial^{\mathbf{I}^{}}_{\mathbf{N}^{}}\circ\partial^{\mathbf{N}^{}}_{\mathbf{I}^{}}:\mathbf{I}\rightarrow\mathbf{I}$ are both multiplication by $V_1+V_2-V_3-V_4.$
\end{lemma}

\begin{proof}
We do $\partial^{\mathbf{I}}_{\mathbf{N}}\circ\partial^{\mathbf{N}}_{\mathbf{I}}:\mathbf{I}\rightarrow\mathbf{I}$ first. 

As in the proof of [OSS Lemma 9.2.4], we proceed by a now-familiar rectangle counting argument. Consider any juxtaposition of rectangles contributing to the left-hand side of the equation $$\partial^{\mathbf{I}^{}}_{\mathbf{N}^{}}\circ\partial^{\mathbf{N}^{}}_{\mathbf{I}^{}} = V_1+V_2-V_3-V_4.$$

This is a rectangle from $\x\in\mathbf{I}(\G_+)$ to $\y\in\mathbf{N}(\G_+)$ (where we count rectangles going through $Y$ markings) and then a rectangle from $\y\in\mathbf{N}(\G_+)$ to $\z\in\mathbf{I}(\G_+)$ (where we count rectangles going through $X$ markings). Thus, $\y$ must not contain $c,$ but $\z$ does. The only possibility is thus that $\x = \z,$ and the composite of these two rectangles must be an annulus. The annulus must be width one since otherwise it would intersect $\X-\{X_1, X_2, Y_1, Y_2\},$ and similarly must be multiplicity 1.

There are 4 annuli, and since they are all thin, they are empty; the $O$-markings they pass through are $O_i$ for $i=1,2,3,4,$ giving $V_1+V_2-V_3-V_4$ since the $V_3$ and $V_4$ terms correspond to vertical annuli whereas the others correspond to horizontal annuli.

The other case uses the same decomposition of rectangles.
\end{proof}

We get the following commutative square:

\begin{tikzcd}[sep = huge]
\mathbf{I}^{'}\arrow{r}{\partial^{\mathbf{N}^{'}}_{\mathbf{I}^{'}}}\arrow{d}{\partial^{\mathbf{N}^{}}_{\mathbf{I}^{}}\circ T} &\mathbf{N}^{'}\arrow{d}{-T\circ\partial^{\mathbf{I}^{'}}_{\mathbf{N}^{'}}}\\
\mathbf{N}\arrow{r}{\partial^{\mathbf{I}^{}}_{\mathbf{N}^{}}} &\mathbf{I} \\
\end{tikzcd}

\begin{lemma}[Modified from \text{[OSS Lemma 9.2.5]}]
Let $M_0, A_0$ be the bigradings on $\G_0$ (whose grid states may be naturally identified with those of $\G_+$) and $M_0,' A_0'$ be the bigradings on $\G_0'$ (whose grid states may be naturally identified with those of $\G_-$.) Endow $\mathbf{I}'$ and $\mathbf{N}'$ with bigradings $M_0'$ and $A_0'+\frac{\ell_0-\ell-1}{2}$ and endow $\mathbf{I}$ and $\mathbf{N}$ with bigradings $M_0+1$ and $A_0+\frac{\ell_0-\ell+1}{2}.$

Then, in the above square, the following holds:
\begin{itemize}
	\item Each edge map is homogenous of bidegree $(-1, 0).$
	\item The left column is isomorphic as a bigraded chain complex over $\mathbb{F}[V_1,\dots, V_n, v]$ to $GCL_S(\G_+)[[-1, 0]].$
	\item The left column is isomorphic as a bigraded chain complex over $\mathbb{F}[V_1,\dots, V_n, v]$ to $GCL_S(\G_-).$
	\item The top row is isomorphic as a bigraded chain complex over $\mathbb{F}[V_1,\dots, V_n, v]$ to $GCL_S(\G_0')[[0,\frac{\ell_0-\ell+1}{2}]].$
	\item The bottom row is isomorphic as a bigraded chain complex over $\mathbb{F}[V_1,\dots, V_n, v]$ to $GCL_S(\G_0)[[-1,\frac{\ell_0-\ell-1}{2}]].$
\end{itemize}
\end{lemma}

\begin{proof}
The proofs of [OSS Lemma 9.2.3] and [OSS Lemma 9.2.5] carry over identically in this situation.
\end{proof}

The following corollary is immediate.

\begin{corollary}
The map $$(-1)^M(\partial^{\mathbf{N}^{}}_{\mathbf{I}^{}}\circ T-T\circ\partial^{\mathbf{I}^{'}}_{\mathbf{N}^{'}}):GCL_S(\G_0')\rightarrow GCL_S(\G_0)$$ is a chain map of $\mathbb{F}[v, V_1, \dots, V_n]$-modules homogeneous of degree $(-2, -1).$
\end{corollary}

Note that the mapping cone of this map is precisely the above commutative square.

\subsection{Defining the New Maps}
Define for $\x \in\mathbf{S}(\G_0),$ the map:
$$P(\x) = \sum_{\y'\in\mathbf{S}(\G_0')}\sum_{\{p\in\Pent^*(\x, \y') \big| p\cap\Y  = \emptyset \}} S(p)v^{\mathcal{R}(p)} V_1^{O_1(p)}\dots V_n^{O_n(p)} \y'.$$

And, for $i=1,2,$ and $\x \in\mathbf{S}(\G_0'),$ the maps:
$$h_{X_2}(\x) = \sum_{\y\in\mathbf{S}(\G_0')}\sum_{\{r\in\Rect^*(\x, \y) \big| r\cap(\Y\cup\X) = X_2\}} S(r)v^{\mathcal{R}(r)} V_1^{O_1(r)}\dots V_n^{O_n(r)} \y.$$

$$h_{Y_i}(\x) = \sum_{\y\in\mathbf{S}(\G_0')}\sum_{\{r\in\Rect^*(\x, \y) \big| r\cap\Y = Y_i, Y_i(r)=1, X_2(r)=0\}} S(r)v^{\mathcal{R}(r)} V_1^{O_1(r)}\dots V_n^{O_n(r)} \y.$$

$$h_{X_2, Y_i}(\x) = \sum_{\y\in\mathbf{S}(\G_0')}\sum_{\{r\in\Rect^*(\x, \y) \big| r\cap\Y = Y_i, Y_i(r)=1, X_2(r)\geq1\}} S(r)v^{\mathcal{R}(r)} V_1^{O_1(r)}\dots V_n^{O_n(r)} \y.$$

Let $h_Y = h_{Y_1} + h_{Y_2}$ and $h_{X_2, Y} = h_{X_2, Y_1} + h_{X_2, Y_2}$.

Further define, for a domain $\psi,$ the quantity $Y(\psi) = Y_1(\psi)+Y_2(\psi).$

\begin{lemma}
The map $P:GCL_S(\G_0)\rightarrow GCL_S(\G_0')$ is a quasi-isomorphism of bigraded chain complexes.
\end{lemma}

\begin{proof}
This is immediate from the commutation invariance of $GCL_S$ as proven in the previous section. Alternatively, the unsigned version of this lemma is proven in [RWW Propositions 9, 10, 11]. (This paper uses a slightly different-looking definition for the map $P$, allowing pentagons to be long without being thin, however these pentagons always contribute 0 since they contain intersections with $\X,$ hence the two maps are in fact identical.)
\end{proof}

\begin{remark}
$h_{X_2, Y_i}$ does not vanish on $\mathbf{I}^{'}.$ This is in stark contrast to the map $h_{X_2, Y_i}$ defined in [OSS Chapter 9] in the un-enhanced case.
\end{remark}

\begin{lemma}
\begin{enumerate}
	\item $h_{X_2}$ vanishes on $\mathbf{N}^{'}$ and maps $\mathbf{I}^{'}$ to $\mathbf{N}^{'}$.
	\item $h_{Y}$ vanishes on $\mathbf{I}^{'}$.
\end{enumerate}
\end{lemma}

\begin{proof}
Both statements follow immediately from the multiplicity conditions on the rectangles in the definitions of the maps.
\end{proof}

\begin{lemma}
\begin{enumerate}
	\item Suppose $r$ is a rectangle contributing to the sum in $h_{X_2}$. Then $r$ is not long, and $Y_1(r)=Y_2(r)=0.$
	\item Suppose $r\in\Rect^*(\x, \y)$ is a rectangle contributing to the sum in $h_{X_2, Y_i}.$ Then, $\y\in\mathbf{N}(\G_0'),$ so as a consequence, the image of $h_{X_2, Y_i}$ is within $\mathbf{N}^{'}.$
\end{enumerate}
\end{lemma}

\begin{proof}
Both statements follow immediately from the multiplicity conditions on the $Y_i$ in the definitions of the maps.
\end{proof}

Now, we must prove two key lemmas:

\begin{lemma}
The following two identities hold: $$(-1)^MP\circ T\circ \partial^{\mathbf{I}^{'}}_{\mathbf{N}^{'}} = h_{X_2}\circ h_Y$$ and $$(-1)^{M+1}P\circ \partial^{\mathbf{N}^{}}_{\mathbf{I}^{}}\circ T = h_{Y}\circ h_{X_2}$$.
\end{lemma}

\begin{proof}
We start with the first equation. Let $\psi = r*t*p$ be a domain contributing to the left-hand side. Here, since the image of $T$ is in $\mathbf{I}^{'},$ a pentagon $p$ contributing to $P$ in the left-hand side has an outgoing corner at $c,$ therefore geometrically we see $X_2(p)=0$ lest $Y_i(p)\geq0$ for either $i=1,2.$ Now, concatenating $p$ with the triangle $t$ contributing to $t$ therefore gives a rectangle $r_2$ with an outgoing corner at $c',$ with $X_2(r_2)=1,$ and $Y_1(r_2)=Y_2(r_2)=0.$ Furthermore, the image of $\partial^{\mathbf{I}^{'}}_{\mathbf{N}^{'}}$ is some rectangle $r=r_1$ with an incoming corner at $c'$, and with $Y(r_1) = 1.$ Hence, the composite $\psi = r_1*r_2$ is a decomposition appearing in the right-hand side. Conversely, for $\psi = r_1*r_2$ a decomposition appearing in the right-hand side, we have that the rectangle $r_1=r$ contributes to $ \partial^{\mathbf{I}^{'}}_{\mathbf{N}^{'}},$ and cutting $r_2$ along the portion of $\beta_i$ passing through $X_2$ gives us a decomposition $r_2 = t*p$ such that $t$ contributes to $T$ and $p$ to $P$ in the left-hand side. Furthermore, these two decompositions always have the same degree since the interiors of the domains in each decomposition differ only on the edge of the small triangle $t$ which is part of $\beta_i$; however, the only point on $\beta_i$ that could possibly affect the degrees of the decompositions is an outgoing corner of $p.$ Hence, either $\mathcal{T}(p)=\mathcal{T}(r_2)=1$ or else $\mathcal{T}(p)=\mathcal{T}(r_2)=0$. Furthermore, by the definition of a sign assignment for a pentagon, it is clear that both decompositions have the same sign.

Now, we consider the second equation. Let $\psi = t*r*p$ be a domain contributing to the left-hand side. We see geometrically that if $X_2(p)\geq1,$ then it must be the case that either $Y_1(p)$ or $Y_2(p)$ is positive. Since we specify in the equation defining $P$ that $Y(p)=0,$ we must have $X_2(p)=0.$ Furthermore, it is not possible for $r$ and $p$ to overlap in a small neighborhood of $c',$ since this would require $Y_1(r)=1$ and then $Y(p)>0.$ Hence, all local multiplicities are $\leq1$ around $c'.$ Suppose $c'$ were a 90-degree corner. Then, we must have $Y_2(r)=1,$ and $p$ must be a left pentagon. However, geometrically, the only way this is possible is if $Y_2(p)\geq1,$ a contradiction.

Also note that $X_2(r) = 0$ by the fact that $c$ is an outgoing corner and $r$ cannot be long. In all cases, then, we have $X_2(\psi) = Y(\psi) = 1.$

Thus, $c'$ is either a 270-degree corner, a 180-degree corner, or in the interior of $\psi$. In all three cases we may represent $\psi$ by an L-shaped region $Q$ in the universal cover such that a pre-image $C'$ of $c'$ is either a 270-degree corner, a 180-degree corner, or in the interior of $\psi$. Indeed, start with $t$ which has a corner at $C',$ and then attach pre-images $R$ and $P$ of $r$ and $p$ such that if $r$ or $p$ has a corner at $c'$ then $R$ or $P$ has a corner at $C'.$

We must also show that these three possibilities are also the only three possibilities for a domain contributing to the right-hand side. Suppose $\psi = r_1*r_2$ where $r_1$ is from $h_{X_2}^+$ and $r_2$ is from $h_Y.$ Then, $r_1$ has a corner at $c',$ namely the southwest corner, and no others since $Y(r_1)=0.$ Then, $r_2$ must contain precisely one of $Y_1$ or $Y_2,$ but $X_2(r_2)=0,$ so either $r_2$ has a corner at $c',$ leaving $\psi$ a 180-degree corner at $c'$ (which clearly forces $\psi$ to be an annulus), or it has an edge passing through $c',$ leaving $\psi$ a 270-degree corner at $c'.$ Note there is no 360-degree case; we will return to this at the end of the proof.

We will first show that in each case $\psi$ has two decompositions of equal degree; then, we will appeal to a proof from [OSS] to show that these two decompositions further have the same sign.

In the first 270-degree case, either $c'$ is on an edge of $r$ or on $p.$ If $c'$ is on an edge of $r,$ which is part of horizontal circle, say, $\alpha_j,$ then cutting $\psi$ along $\alpha_j$ gives a decomposition $\psi = r_1*r_2.$ Since $r_1$ clearly has $c$ as an outgoing corner and contains $X_2,$ and we know $X_2(\psi)=X_2(r)+X_2(t)+X_2(p)=0+1+0=1,$ that tells us $X_2(r_1)=1.$ Furthermore, in this case the support of $r_1$ is a subset of the support of $p,$ hence $Y(r_1)=0.$ Thus, $r_1$ contributes to $h_{X_2}^+$; likewise, we must have $Y(r_2)=Y(\psi)=1,$ so $Y$ contributes to $h_Y^+.$ Conversely, if $\psi = r_1*r_2,$ with $Y_1(r_2)=1,$ then performing this same cut along $\alpha_j$ gives a decomposition as in the left-hand side. 

The proof of lemma \ref{rectangles3} ensures that the two decompositions $r_1*r_2$ and $r_1'*r_2'$ of $\psi$ as rectangles in $GC^+(\G_0')$ have the same degree. One of these decompositions is $r_1*r_2,$ and the other $p*r*t$ differs from $r_1'*r_2'$ by the introduction of the small triangle $t$ and by cutting at $\beta_i$ instead of $\gamma_i$. No points on $\beta_i$ or $\gamma_i$ can contribute to the degrees of any of these decompositions, for on the left-hand side, $p$ and $r$ border $\beta_i$ and on the right-hand side, $r_1$ and $r_2$ border $\gamma_i.$ Thus, both decompositions $r_1*r_2$ and $p*r*t$ have the same degree.

Similarly, if $c'$ is on an edge of $p,$ then this edge is part of $\gamma_i,$ and we cut $\psi$ along $\gamma_i$ to get a decomposition $r_1*r_2.$ In this case, $r_1$ is the union of $t$ and a portion of $r$ not containing $Y_2$; hence, $X_2(r_1)=1$ and $Y(r_1)=0,$ so $r_1$ contributes to $h_{X_2}^+.$ Similarly, $r_2$ must contribute to $h_Y^+$ since we must have $Y(r_2)=1$. Conversely, if $\psi = r_1*r_2,$ with $Y_2(r_2)=1,$ then performing this same cut along $\gamma_i$ gives a decomposition as in the left-hand side. Likewise, these two decompositions differ  by the introduction of the small triangle $t$ and by cutting at $\beta_i$ instead of $\gamma_i$, so they also must have the same degree by the argument in the previous paragraph.

In the 180-degree case on the left-hand side, we must have $Y_2(r)=1,$ and since $r$ is not long, the condition that there is a 180-degree corner at $c'$ forces $p$ to have an incoming corner at $c',$ and also for $\x = \z$; hence, the region $\psi$ must be a horizontal annulus, and it must be thin to avoid intersecting with $\Y-\{Y_1, Y_2\}$. Furthermore, since the local multiplicities of $c',$ which is an incoming corner of $p,$ are all $\leq1,$ we must have that the annulus has multiplicity 1. There is clearly a unique decomposition of this annulus into two rectangles $r_1, r_2$ such that $X_2(r_1)=Y(r_2)=1$ and $Y(r_1)=X_2(r_2)=0.$ Conversely, if $\psi = r_1*r_2$ has a 180-degree corner at $c',$ either we are in the horizontal case, in which case there is a unique decomposition $\psi = t*r*p$ as in the previous paragraph, or else we are in the vertical case (see the following paragraph); this case is unique for each grid state $\x$. Since the annulus is thin, all constituent polygons must be empty so the degrees of all decompositions are 0.

There is one case remaining for the left-hand side and one case remaining for the right-hand side. For the left-hand side, this case is that $c'$ is a 360-degree corner, in which case we must have $Y_1(r)=1$ and therefore that $c'$ is an incoming corner at $p$; again, this forces the initial and final grid states to be equal, so we have a vertical annulus, and it is unique for each grid state $\x.$ The remaining case for the right-hand side is a vertical annulus as well. Both annuli are thin, so have degree 0, and contribute 1 power of $V_4$ to the formula. 

What's left to show is that each decomposition in every case has the same sign. Since each of these possibilities is identical to those in the proof of [OSS Theorem 15.5.1], this proof guarantees that each decomposition has the same sign. Indeed, it only uses the defining properties of sign assignments and the relation between the sign of a pentagon and its associated rectangle, all of which carry over to the case where the pentagons and rectangles are potentially long.

This concludes the proof.
\end{proof}

\begin{lemma} The map $h_{X_2, Y}$ provides a chain homotopy from $h_{X_2}\circ h_Y + h_{Y}\circ h_{X_2}$ to multiplication by $V_2-V_4$, i. e. the following equation holds:
$$h_{X_2}\circ h_Y + h_{Y}\circ h_{X_2} + h_{X_2, Y}\circ \partial + \partial\circ h_{X_2, Y} = V_2 - V_4.$$
\end{lemma}

\begin{proof}
Let $\psi\in\pi(\x, \z)$ be a region contributing to the left-hand side of this equation with $|D(\psi)|>0.$ We have three cases, $|\x - (\x\cap \z)| = $0, 3, or 4.

If $|\x - (\x\cap \z)| = 4,$ then Lemma \ref{rectangles4} shows that $|D(\psi)|=2,$ and both decompositions have the same degree and opposite signs. Hence, this case clearly contributes 0 to the sum on the left-hand side.

If $|\x - (\x\cap \z)| = 3,$ then Lemma \ref{rectangles3} shows that $|D(\psi)|=2,$ and both decompositions have the same degree and opposite signs, since only 1 rectangle in each possible decomposition of $\psi$ could be long. We must show that both decompositions appear in the above formula, and do so exactly once. By analyzing the above formula, we conclude that $Y(\psi) = 1$ and $X_2(\psi)\geq1.$ Both rectangles in both decompositions clearly have trivial intersections with $\Y - \{Y_1, Y_2\}$. Hence, the only possibilities can be found by considering multiplicities of $X_2$ and $Y_i,$ and are listed below. Note that in the first two cases, $r_1$ is not forced; it could be from either $h_{X_2}$ or $\partial,$ but this choice is forced when we consider $r_2.$

\begin{tabular}{ c | c | c | c | c }
$X_2(\psi)$ & $X_2(r_1)$ & $Y(r_1)$ & $r_1$ & $r_2$ \\
\hline
\hline
\hline
1 & 1 & 0 & $h_{X_2}$ & $h_Y$ \\
\hline
$>1$ & 1 & 0 & $\partial$ & $h_{X_2,Y}$ \\
\hline
$\geq1$ & 1 & 1 & $h_{X_2, Y}$ & $\partial$ \\
\hline
$\geq1$ & 0 & 0 & $\partial$ & $h_{X_2,Y}$ \\
\hline
$1$ & 0 & 1 & $h_Y$ & $h_{X_2}$ \\
\end{tabular}

Finally, I claim the case $X_2(\psi)>1$, $X_2(r_1)=0$, and $Y(r_1)=1$ cannot occur. Indeed, this forces the other rectangle $r_2$ to satisfy $X_2(r_2)>1$ but also $Y(r_2)=0$ which is impossible.

Hence, this case $|\x - (\x\cap \z)| = 3$ also contributes 0 to the sum.

If $|\x - (\x\cap \z)| = 0,$ then we must have $\psi$ is an annulus containing $X_2,$ and it must be thin since $Y(\psi)=1$ and $\psi\cap(\Y - \{Y_1, Y_2\}) = \emptyset.$ Since $Y(\psi)=1,$ it must have multiplicity 1. Then, clearly, $\psi$ has a unique decomposition into two rectangles, and the possibilities in the table in the previous case still apply to show that this decomposition appears precisely once in the sum on the left-hand side. If the annulus is horizontal, it gives us multiplication by $V_2,$ and if it is vertical, we get multiplication by $-V_4.$
\end{proof}

The remainder of the proof of the skein exact sequence proceeds exactly as in [OSS Chapter 9], which gives us the theorem. $\Box$

\section{Alternatives to $\tau$}

In this section, we discuss constructions relating to and generalizing the $\tau$ invariant. We first recall the definition of $\tau.$

\begin{definition}
For a knot $K,$ $\tau(K)$ is -1 times the maximum Alexander grading of a homogeneous nontorsion element in $GH^-(K).$
\end{definition}

Note that for the remainder of the paper, the coefficient field will be $\mathbb{F}$ unless otherwise stated. We use a field for coefficients to maximize simplicity here, although much of the below can be generalized to $\mathbb{Z}$ coefficients without too much difficulty.

\subsection{The endomorphism $\partial_1$}

[L page 9] briefly mentions the fact that $\partial_1$ is a chain map. Let's recall why. The differential is defined as: $\partial(\x) = \sum_{k=0}^\infty v^k\partial_k(\x).$ In proving that $\partial^2 = 0,$ we may expand out this sum and conclude that each coefficient of $v^k$ is 0. This gives: $$\partial_0^2 = 0,$$ $$\partial_0\partial_1+\partial_1\partial_0 = 0,$$ $$\partial_0\partial_2+\partial_1^2+\partial_2\partial_0 = 0,$$ and so on.

The second equality tells us that $\partial_1:GC^-(\G)\rightarrow GC^-(\G)$ is a chain map on the unblocked grid homology, so it induces a homomorphism $\partial_{1*}:GH^-(\G)\rightarrow GH^-(\G).$ Furthermore, the third equation tells us that $\partial_1^2$ is nullhomotopic, hence $\partial_{1*}$ is a differential on $GH^-(\G)$ and we can consider its homology. 

The conjecture is that $\partial_{1*}$ is the 0 homomorphism, but we have yet to find a nulhomotopy for it. The author has attempted to try using certain $L$-shaped regions as a count; this fails.

\begin{theorem}
The endomorphism $\partial_{1*}:GH^-(\G)\rightarrow GH^-(\G)$ is a knot invariant (i.e., for a given knot, it does not depend on the choice of grid presentation $\G$.)
\end{theorem}

\begin{proof}
By Theorem \ref{cromwell}, it is sufficient to show that $\partial_{1*}$ is preserved under commutation and SW-destabilization maps.

For the commutation maps, let $\G$ and $\G'$ differ by a commutation. We know by our proof of invariance that there is a chain map $P_S:(GCL_S(\G), \partial)\rightarrow(GCL_S(\G'), \partial)$ such that the coefficient of $v^0$ of $P_S$, when restricted modulo 2, is precisely the quasi-isomorphism $P:(GC^-(\G), \partial_0)\rightarrow(GC^-(\G'), \partial_0)$ from [OSS] Section 5.1. Let $P_L$ be the restriction of $P_S$ modulo 2. Let $P_1$ be the $v^1$ term of $P_L.$ Then, the $v^1$ term of the equation $$\partial\circ P_L + P_L\circ\partial = 0,$$ which expresses that $P_L$ is a chain map, is: $$\partial_0\circ P_1 + P_1\circ\partial_0+\partial_1\circ P + P\circ\partial_1 = 0,$$ which tells us that $\partial_1$ commutes with $P$ up to homotopy. Hence, $\partial_{1*}$ is preserved under commutation.

For SW-destabilization, if $\G'$ differs from $\G$ by a SW-destabilization, [OSS Lemma 5.2.17] gives us a quasi-isomorphism from $GC^-(\G)$ to the mapping cone of the function $V_2-V_1:GC^-(\G')[V_1][[1,1]]\rightarrow GC^-(\G')[V_1].$ Using notation from [OSS Section 5.2], the quasi-isomorphism is given by $(i,n)\mapsto (e(i), e(\mathcal{H}_{X_2}^{\mathbf{I}}(n))),$ for $(i, n)\in\mathbf{I}\oplus\mathbf{N}$. The map $e$ clearly commutes with $\partial_1,$ and the map $\partial_1$ commutes with $\mathcal{H}_{X_2}^{\mathbf{I}}$ up to homotopy by a similar argument to the previous paragraph. Indeed, $\mathcal{H}_{X_2}^{\mathbf{I}}$ is the part of the homotopy operator $\mathcal{H}_{X_2}$, which is the $v^0$ term of the chain map $\mathcal{H}_{X_2}:GCL_S(\G)\rightarrow GCL_S(\G),$ when taken modulo 2.

Finally, the isomorphism from the homology of the mapping cone of the function $V_2-V_1:GC^-(\G')[V_1][[1,1]]\rightarrow GC^-(\G')[V_1]$ to $GH^-(\G')$ commutes with the induced map of $\partial_1.$ Indeed, this isomorphism is proved in [OSS, Lemma 5.2.16] as an isomorphism of $\mathbb{F}[V_2, \dots, V_n]$-modules; this identical proof shows we have an isomorphism of $\mathbb{F}[V_2, \dots, V_n, \partial_1]$-modules.
\end{proof}

\begin{corollary}
The homology of the chain complex $(GH^-(\G), \partial_{1*})$ is a knot invariant.
\end{corollary}

\begin{lemma}
The image of $\partial_{1*}$ is $U$-torsion.
\end{lemma}

\begin{proof}
By definition, $\partial_{1*}$ commutes with $U$. Also, $\partial_{1*}$ lowers the grading by $(-3, 0).$ 

Suppose that $\xi\in GH^-(K)$ was such that $\partial_{1*}(\xi)$ is $U$-nontorsion, then $\xi$ must be $U$-torsion since the $U$-nontorsion elements of $GH^-(K)$ are all supported in a line of slope $d-2s=0.$ Thus, there exists $k$ such that $U^k\xi=0,$ so $$U^k\partial_{1*}(\xi)=\partial_{1*}(U^k\xi)=\partial_{1*}(0)=0,$$ a contradiction.
\end{proof}

Since we know that the $U$-torsion elements of $GH^-(K)$ form a finite-dimensional vector space, we now extract some more concrete invariants from $\partial_{1*}$:

\begin{itemize}
	\item The bigraded $\mathbb{F}$-vector space $\partial_{1*}(GH^-(K))$
	\item The maximum $k$ such that $U^k(\partial_{1*}(GH^-(K)))$ is nontrivial.
	\item The maximum $k$ such that $\partial_{1*}(U^k\zeta)\neq0,$ where $\zeta\in GH^-(K)$ is a homogeneous $U$-nontorsion element of grading $(-2\tau, -\tau)$.
\end{itemize}

\subsection{$\partial_{1*}$ for $\widehat{GH}(K)$}

\begin{lemma}
The map $\partial_1$ induces a map $\partial_{1*}:\widehat{GH}(\G)\rightarrow\widehat{GH}(\G)$ which is a well-defined knot invariant.
\end{lemma}

\begin{proof}
Let $\pi_i:GC^-(\G)\rightarrow GC^-(\G)/V_i$ be the projection map.

From the previous subsection, it is sufficient to show that $\partial_1$ commutes with isomorphisms from $H(GC^-(\G)/V_i)$ to $H(GC^-(\G)/V_j)$ (since then the arguments that $\partial_{1*}$ is a knot invariant would then apply immediately.)

Each such isomorphism comes from composing quasi-isomorphisms $\text{Cone}(V_i)\rightarrow GC^-/V_i$ induced by projection $(c, c')\mapsto \pi_i(c'),$ and isomorphisms $\text{Cone}(V_i)\rightarrow\text{Cone}(V_j)$ given by $(c, c')\mapsto(c, \mathcal{H}(c)+c'),$ where $\mathcal{H}$ is the homotopy operator from $V_i$ to $V_j.$

By the previous subsection, we know that $\partial_{1}$ commutes with all of these maps up to homotopy (when we consider the action of $\partial_{1}$ on the mapping cone as $\partial_1(c, c')=(\partial_1c, \partial_1c')$), which establishes the result.
\end{proof}

Now, from [OSS Chapter 7.1] we know that $\widehat{GH}(m(K)),$ where $m(K)$ is the mirror of the knot $K,$ is canonically isomorphic to the dual vector space $\widehat{GH}(K)^{\vee}.$ Repeating our entire above discussion with sign-refined, we may use the universal coefficient theorem to give ourselves a coefficient field of $\mathbb{R}$ rather than $\mathbb{F}$ (which we will denote as $\widehat{GH}(K; \mathbb{R}).$

\begin{lemma}
If $K$ is an amphicheiral knot, then there exists at least one nondegenerate bilinear form $\langle\cdot, \cdot\rangle$ on $\widehat{GH}(K; \mathbb{R})$ for which $\partial_{1*}$ is self-adjoint and the image of $\partial_{1*}$ is an isotropic subspace.
\end{lemma}

\begin{proof}
We suppress the $\mathbb{R}$-coefficients for convenience here. The invariance of $\widehat{GH}$ on the grid presentation of a knot gives a bigraded isomorphism $\Omega:\widehat{GH}(K)\rightarrow\widehat{GH}(m(K))\cong\widehat{GH}(K)^{\vee},$ which in turn induces such bilinear form $$\langle x, y\rangle:= \Omega(x)(y).$$ Following the isomorphisms in [OSS Chapter 7], we see that $(\Omega(x))(\partial_{1*}y) = \Omega(\partial_{1*}x)(y),$ which proves the self-adjoint claim. Since $\partial_{1*}^2=0,$ we get that the image of $\partial_{1*}$ is an isotropic subspace.
\end{proof}

\begin{remark}
We hope that, perhaps, we could prove that $\partial_{1*}$ is zero by computing the signature of one such bilinear form and showing it is positive definite, say; this would require the bilinear form to be symmetric, which is a difficult question. It is another question of interest whether this bilinear form depends on the particular isotopy of $K$ into $m(K).$

When $K$ is alternating, say, then we know that $\widehat{GH}$ is dimension $\leq1$ in each bigrading, hence the fact that $\Omega$ is bigraded forces each such bilinear form to be diagonal with respect to the basis consisting of nonzero homogeneous elements, hence symmetric.
\end{remark}

\begin{remark}
We still have that $\partial_{1*}$ changes the grading $(Maslov, Alexander)$ by $(-3, 0).$ A python search gives that 18 is the smallest crossing number of a prime knot with $\widehat{GH}$ nonzero in gradings differing by $(-3, 0)$.
\end{remark}

\subsection{The connection between $GHL(K)$ and $\tau$}

\begin{theorem}
There is a spectral sequence $\{E_r, d_r\}$ where $E_2 = GH^-(\G)[v]$ and $d_2([\x]) = v\partial_{1*}([\x])$, which converges to $GHL(\G)$ as $\mathbb{F}[U]$-modules. The terms $\{E_n, d_n\}$ for $n\geq1$ are knot invariants.
\end{theorem}

\begin{proof}
This follows immediately from the fact that $GCL(\G) = \oplus_{n=0}^\infty v^nGC^-(\G),$ which is an expression of $GCL$ as the associated graded object of the filtration $\mathcal{F}^pGC^+(\G) = \oplus_{n=p}^\infty v^nGC^-(\G).$ Clearly, this filtration is respected by $\partial.$ Furthermore, unwinding definitions, it is clear that $d_2([\x]) = v\partial_{1*}([\x]).$ 

The spectral sequence converges by grading reasons. Indeed, each differential $d_r$ increases the filtration by $r-1$ but decreases the Maslov grading by $2r-1$; since the set of Maslov gradings for each fixed Alexander grading is finite, we must have convergence.

The fact that $\{E_n, d_n\}$ for $n\geq1$ are knot invariants follows from the fact that the maps on $GH^+$ used to prove that $GH^+$ is a knot invariant all preserve the filtration (see [McC Chapter 3]).
\end{proof}

\begin{remark}
Unfortunately, we lose the structure of the $v$ map in this spectral sequence.
\end{remark}

\begin{remark}
This is the same proof strategy used in Rasmussen's proof of the invariance of the Lee-Rasmussen spectral sequence in [R10].
\end{remark}

This prompts us to define many alternatives to $\tau.$

\begin{definition}
Let $K$ be a knot.
\begin{itemize}
	\item $\tau^+(K)$ is -1 times the maximum Alexander grading of a $U$- and $v$-nontorsion homogeneous element in $GHL(K).$
	\item $\tau^+_U(K)$ is -1 times the maximum Alexander grading of a $U$-nontorsion homogeneous element in $GHL(\G).$
	\item $\rho(K)$ is the maximum $k$ such that the equation $U^k\xi = 0$ has a nonzero solution $\xi\in GHL(K).$
\end{itemize}
\end{definition}

From the spectral sequence and the definition, we immediately get the following lemma:

\begin{lemma}
For any knot $K,$ $-\tau^+(K)\leq -\tau^+_U(K)\leq -\tau(K).$
\end{lemma}

\begin{remark}
As in [OSS Section 7.4], mirroring the knot $K$ essentially dualizes the complex $GCL(K)$ over the ring $\mathbb{F}[U, v],$ and it would be very convenient to use this fact to prove that $-\tau^+(K)=\tau^+(m(K))$, which in combination with the above lemma and the similar fact about $\tau,$ would show that $\tau^+(K)=\tau^+_U(K)=\tau(K).$ However, $\mathbb{F}[U, v]$ is not a principal ideal domain, and the complications of the universal coefficient spectral sequence render this line of proof quite difficult.
\end{remark}

\section{Example Computations with the Spectral Sequence}

The spectral sequence defined in the previous section allows us to compute the homology $GHL(K)$ for many families of knots $K$ by first computing $GH^-(K)[v].$ In this section, we compute the examples of alternating knots and torus knots; in both examples the spectral sequence collapses at the $E_2$ page, so the computation is especially simple. The computations are therefore purely algebraic, not requiring any more topological information about the knots than the structure of $GH^-(K).$

This provides a bit of evidence for the conjecture that $GHL(K)\cong GH^-(K)[v]$ always, and allows us to see examples  of what $GHL(K)$ looks like in practice.

\begin{theorem}\label{alt}
If $K$ is a quasi-alternating knot, then $GHL(K) \cong GH^-(K)[v]$ as bigraded $\mathbb{F}[U]$-modules.
\end{theorem}

\begin{proof}
Let $K$ be quasi-alternating. It is sufficient to show that the spectral sequence with $E_2$ page $GH^-(K)[v]$ converging to $GHL(K)$ collapses at the $E_2$ page.

By [OSS Chapter 10], we know that the $U$-torsion part of the grid homology $GH^-(K)$ is supported in bigradings $(M,A)$ with $M-A=-\tau,$ and the $U$-nontorsion ``tail'' is supported in $M-2A =0,$ which begins on the line $M=A$ and extends in one direction.

Hence, $GH^-[v](K),$ which is what we get when we first take $GC^-[v](K)$ and take homology with respect to $\partial_0,$ must be supported in copies of these shapes each differing from the previous by an addition of 2 in Maslov grading. In particular, all homogeneous $U$-torsion elements have the same parity of their Maslov grading.

We now wish to conclude, inductively, that all the higher differentials in the spectral sequence are trivial.

If $\xi\in GH^-[v](K)$ is $U$-torsion, then so is $d_i(\xi)$ for any $i$ by $U$-equivariance of the differential. Thus, any torsion elements must be in the kernel of $d_i$ for each $i$ since $d_i$ reverses the parity of the Maslov grading.

Now, note that each differential in the spectral sequence increases the coefficient of $v$ strictly. But, any point on the tail $M-2A=0$ is strictly lower in Maslov grading than any point of equal Alexander grading with a higher $v$ coefficient. Thus, the differentials must all be zero there as well.
\end{proof}

\begin{theorem}\label{tor}
If $K$ is a torus knot, then $GHL(K) \cong GH^-(K)[v]$ as bigraded $\mathbb{F}[U]$-modules.
\end{theorem}

\begin{proof}
Again, we show the spectral sequence collapses at $E_2,$ and we do so by analyzing the $U$-torsion and $U$-non-torsion portions separately. Let $K$ be a positive torus knot $K=T_{p,q}.$ [OSS Theorem 16.2.6] describes the structure of $\widehat{GC}(K)$. In particular, it shows us, by the computation that $\tau(K)=\frac{(p-1)(q-1)}2,$ that the infinite tail begins at $\mathbb{F}_{(\delta_{-k}, n_{-k})},$ where $n_{-k} = -\frac{(p-1)(q-1)}2.$ 

Per [OSS Formula 7.6], we can express $GH^-(K) = \mathbb{F}[U]_{(-2\tau, -\tau)}\oplus_i \mathbb{F}[U]_{(d_i, s_i)}/U^{n_i}$. Let $\xi_i$ be the nontrivial element in $\mathbb{F}[U]_{(d_i, s_i)}/U^{n_i}$ that is not a multiple of $U.$ Then, the $\xi_i$ generate the $U$-torsion portion of $GH^-(K).$ Call each such $\mathbb{F}[U]_{(d_i, s_i)}/U^{n_i}$ a ``finite tail.'' 

Furthermore, any finite tails in $GC^-(K)$ contribute precisely two terms to $\widehat{GC}(K)$ per [OSS Formula 7.6], and by this formula, we see that these finite tails must lie in higher Alexander gradings than the tail. Then, clearly any differential $d_i,$ for $i\geq2$, which increases the coefficient of $v$ strictly, cannot map any point on the infinite tail to any nontrivial point. Hence it is sufficient to consider the $U$-torsion parts.

If the map $d_2$ is nontrivial on the $U$-torsion part, then some component of it maps some homogeneous $U$-torsion element $\xi = U^k\xi_i$ to a point $\zeta$ with 1 lower Maslov grading and identical Alexander grading, which is some $U, v$-linear combination of the $\xi_j.$ By grading reasons, none of the $\xi_j$ contributing nonzero terms in the expression of $\zeta$ is $\xi_i.$ The $U$-equivariance of $d_1$ tells us that at least one tail containing some $\xi_j$, $j\neq i,$ must terminate at the same Alexander grading $A^*$ at which the tail containing $\xi_i$ terminates. But, this requires there to be dimension $\geq2$ of $\widehat{GC}(K)$ at the grading $A=A^*,$ contradicting Theorem 16.2.6. Inductively, we may use this same argument to show $d_k$ is trivial on the $U$-torsion part for $k>2.$ By the construction of the spectral sequence we must have that the sequence collapses on $E_2,$ as desired.

A similar computation holds for negative torus knots, based on the fact that for a negative torus knot, a $U$-non-torsion homogeneous element of $GH^-(K)$ never lies in the same Alexander grading but greater Maslov grading as a $U$-torsion homogeneous element of $GH^-(K)$.
\end{proof}

\begin{remark}
These results can likely be extended without much difficulty to other classes of knots with relatively thin knot Floer homology.
\end{remark}

\section{Conclusion}

It still remains open whether the homology $GHL_S(K;\mathbb{Z})$ encodes different information than $GH_S^-(K; \mathbb{Z})[v]$. This article shows that these two objects at least obey very similar topological properties. Strategies for perhaps exhibiting an isomorphism between these two objects may relate to using the mirror of a knot, as discussed in the previous section, which may prove effective at least in showing that the non-torsion parts of these objects are isomorphic. Some of the variants of $\tau$ for $GHL_S(K;\mathbb{Z})$ are possibly sharper topological invariants than $\tau$, so if these homologies are in fact different, we can extract useful data from our ventures.

One pressing open question is to explore if there is an analog of the filtered theory discussed in [OSS Chapters 13 and 14] for double-point enhanced grid homology. The seemingly natural extension of the differential in this case to the double-point enhanced world is markedly not a differential anymore. If in fact the homology $GHL_S(K;\mathbb{Z})$ is isomorphic to $GH_S^-(K; \mathbb{Z})[v]$, then we would expect some filtered theory to exist.

\section{References}

\hspace{15pt}[L] R, Lipshitz, Heegaard Floer Homology, Double Points, and Nice Diagrams. \emph{New
perspectives and challenges in symplectic field theory}, CRM Proc. Lecture
Notes, vol. 49, pp. 327-342. AMS, Providence, RI, 2009.
\medskip

[McC] J. McCleary, \emph{User's Guide to Spectral Sequences}. Mathematics Lecture Series, 12. Publish or Perish Inc, 1985.
\medskip

[MOS] C. Manolescu, P. Ozsv\'{a}th, and S. Sarkar. A combinatorial description of knot Floer homology. \emph{Annals of Mathematics}, 169(2), pp. 633-660, 2009.
\medskip

[MOST] C. Manolescu, P. Ozsv\'{a}th, Z. Szab\'{o}, and D. Thurston. On combinatorial link Floer homology. \emph{Geom. Top.}, 11, pp. 2339-2412, 2007.
\medskip

[OS] P. Ozsv\'{a}th,  Z. Sz\'{a}bo. Holomorphic disks and knot invariants. \emph{Adv. Math.}, 186(1), pp. 58-116, 2004.
\medskip

[OSS] P. Ozsv\'{a}th, A. Stipcisz, and Z. Sz\'{a}bo. \emph{Grid Homology for Knots and Links}. Mathematical Surveys and Monographs, Vol. 208, AMS 2015.
\medskip

[R03] J. Rasmussen, \emph{Floer homology and knot complements.} PhD thesis, Harvard University, 2003.
\medskip

[R10] J. Rasmussen, Khovanov homology and the slice genus. \emph{Inv. Math.}, 182(2), pp. 419-447, 2010.
\medskip

[RWW] T. Ratigan, J. Wang, and L. Wang. A combinatorial proof of invariance of double-point enhanced grid homology. math.GT/1810.03202, 2018.

\end{document}